\documentclass[11pt,oneside,reqno]{amsart}
\usepackage[all]{xy}
\usepackage{amsmath,amsthm,amsfonts,amssymb,times}
\usepackage[mathscr]{eucal}
\usepackage{verbatim}
\usepackage{url}
\setbox0=\hbox{$+$}
\newdimen\plusheight
\plusheight=\ht0
\def\+{\;\lower\plusheight\hbox{$+$}\;}

\setbox0=\hbox{$-$}
\newdimen\minusheight
\minusheight=\ht0
\def\-{\;\lower\minusheight\hbox{$-$}\;}

\setbox0=\hbox{$\cdots$}
\newdimen\cdotsheight
\cdotsheight=\plusheight
\def\cds{\lower\cdotsheight\hbox{$\cdots$}}

\makeatletter
\def\leqalignno#1{\displ@y \tabskip\z@ plus\@ne fil
  \halign to\displaywidth{\hfil$\@lign\displaystyle{##}$\tabskip\z@skip
    &$\@lign\displaystyle{{}##}$\hfil\tabskip\z@ plus\@ne fil
    &\kern-\displaywidth\rlap{$\@lign\hbox{\rm##}$}\tabskip\displaywidth\crcr
    #1\crcr}}
\makeatother

\newcommand{\eb}{\begin{equation}}
\newcommand{\ee}{\end{equation}}
\renewcommand{\Im}{\operatorname{Im}}
\newcommand{\df}{\dfrac}
\newcommand{\tf}{\tfrac}

\renewcommand{\Re}{\operatorname{Re}}
\renewcommand{\Im}{\operatorname{Im}}

 \renewcommand{\a}{\alpha}
\renewcommand{\b}{\beta}

\renewcommand{\d}{{\delta}}

\newcommand{\G}{\Gamma}
\renewcommand{\l}{\lambda}

\renewcommand{\t}{\tau}

\renewcommand{\Re}{\textup{Re}}
\renewcommand{\Im}{\textup{Im}}

\renewcommand{\(}{\left\(}
\renewcommand{\)}{\right\)}
\renewcommand{\[}{\left\[}
\renewcommand{\]}{\right\]}

\numberwithin{equation}{section}
 \theoremstyle{plain}
\newtheorem{theorem}{Theorem}[section]
\newtheorem{lemma}[theorem]{Lemma}
\newtheorem{corollary}[theorem]{Corollary}

\newtheorem{remark}[theorem]{Remark}

\allowdisplaybreaks

\begin{document}
\title[Arithmetical identities]
{Two General Series Identities Involving Modified Bessel Functions and a Class of Arithmetical Functions}
\author{Bruce C.~Berndt, Atul Dixit, Rajat Gupta, Alexandru Zaharescu}
\thanks{2020 \textit{Mathematics Subject Classification.} Primary 33C10; Secondary 11M06, 11N99.\\
\textit{Keywords and phrases.} Bessel functions, functional equations, classical arithmetic functions}
\address{Department of Mathematics, University of Illinois, 1409 West Green
Street, Urbana, IL 61801, USA} \email{berndt@illinois.edu}
\address{Department of Mathematics, Indian Institute of Technology Gandhinagar, Palaj, Gandhinagar 382355, Gujarat, India}\email{adixit@iitgn.ac.in}
\address{Department of Mathematics, Indian Institute of Technology Gandhinagar, Palaj, Gandhinagar 382355, Gujarat, India}\email{rajat\_gupta@iitgn.ac.in}
\address{Department of Mathematics, University of Illinois, 1409 West Green
Street, Urbana, IL 61801, USA; Institute of Mathematics of the Romanian
Academy, P.O.~Box 1-764, Bucharest RO-70700, Romania}
\email{zaharesc@illinois.edu}

\begin{abstract}
    We consider two sequences $a(n)$ and $b(n)$, $1\leq n<\infty$, generated by Dirichlet series
    $$\sum_{n=1}^{\infty}\df{a(n)}{\lambda_n^{s}}\qquad\text{and}\qquad
    \sum_{n=1}^{\infty}\df{b(n)}{\mu_n^{s}},$$
        satisfying a familiar functional equation involving the gamma function $\Gamma(s)$. Two general identities are established.  The first involves the modified Bessel function $K_{\mu}(z)$, and can be thought of as a `modular' or `theta' relation wherein modified Bessel functions, instead of exponential functions, appear.  Appearing in the second identity are $K_{\mu}(z)$, the Bessel functions of imaginary argument $I_{\mu}(z)$, and ordinary hypergeometric functions ${_2F_1}(a,b;c;z)$.  Although certain special cases appear in the literature, the general identities are new.  The arithmetical functions appearing in the identities include Ramanujan's arithmetical function $\tau(n)$; the number of representations of $n$ as a sum of $k$ squares $r_k(n)$; and primitive Dirichlet characters $\chi(n)$.
    \end{abstract}
\maketitle
\section{Introduction}

Our goal is to establish two general identities involving arithmetical functions whose generating functions are Dirichlet series satisfying Hecke's functional equation. For example, two of these arithmetical functions are $r_k(n)$, the number of representations of $n$ as a sum of $k$ squares, and Ramanujan's arithmetical function $\tau(n)$.  Our general theorems involve the Bessel function of imaginary argument $I_{\nu}(z)$ and the modified Bessel function $K_{\nu}(z)$, defined, respectively, in \eqref{b1} and \eqref{b2} below.

One of the identities is a modular or theta relation in which, roughly, the exponential functions are replaced by modified Bessel functions.  The other is a transformation formula in which ordinary hypergeometric functions appear on one side.  Certain special cases, which we cite in the sequel, of each of the two primary identities have appeared in the literature.  However, the general theorems and the majority of the examples are new.

We consider the class of arithmetical functions studied by K.~Chandrasekharan and R.~Narasimhan \cite{cn1}.
Let $a(n)$ and $b(n)$, $1\leq n<\infty$, be two sequences of complex numbers, not identically 0.  Set
\begin{equation}\label{18}
\varphi(s):=\sum_{n=1}^{\infty}\df{a(n)}{\lambda_n^s}, \quad \sigma>\sigma_a; \qquad
\psi(s):=\sum_{n=1}^{\infty}\df{b(n)}{\mu_n^s}, \quad \sigma>\sigma_a^*,
\end{equation}
where  throughout our paper, $\sigma=\Re(s)$, $\{\l_n\}$ and $\{\mu_n\}$ are two sequences of positive numbers, each tending to $\infty$, and $\sigma_a$ and $\sigma_a^*$ are the (finite) abscissae of absolute convergence for $\varphi(s)$ and $\psi(s)$, respectively.  Assume  that $\varphi(s)$ and $\psi(s)$ have analytic continuations into the entire complex plane $\mathbb{C}$ and are analytic on $\mathbb{C}$ except for a finite set $\bf{S}$ of poles.
Suppose that for some $\delta>0$, $\varphi(s)$ and $\psi(s)$  satisfy a functional equation of the form
\begin{equation}\label{19}
\chi(s):=(2\pi)^{-s}\Gamma(s)\varphi(s)=(2\pi)^{s-\delta}\Gamma(\delta-s)\psi(\delta-s).
\end{equation}
Chandrasekharan and Narasimhan proved that the functional equation \eqref{19} is equivalent to Theorems \ref{thmmodular} and \ref{thmriesz} below \cite[p.~6, Lemmas 4, 5]{cn1}, the first of which is due to Bochner \cite{bochner}.  Hence, the validity of any one of \eqref{19}, Theorem \ref{thmmodular}, and Theorem \ref{thmriesz} implies the truth of the other two identities.

\begin{theorem}\label{thmmodular} The functional equation \eqref{19} is equivalent to the `modular' relation
\begin{equation}\label{modular}
\sum_{n=1}^{\infty}a(n)e^{-\lambda_n x}=\left(\df{2\pi}{x}\right)^{\delta}\sum_{n=1}^{\infty}b(n)e^{-4\pi^2\mu_n/x}+P(x), \qquad \Re(x)>0,
\end{equation}
where
\begin{equation*}
P(x):=\frac{1}{2\pi i}\int_{\mathcal{C}}(2\pi)^z\chi(z)x^{-z}dz,
\end{equation*}
where $\mathcal{C}$ is a curve or curves encircling all of $\bf{S}$.
\end{theorem}

Recall that the ordinary Bessel function $J_{\nu}(z)$ is defined by \cite[p.~40]{watson}
\begin{equation*}
J_{\nu}(z):=\sum_{n=0}^{\infty}\df{(-1)^n\left(\frac12 z\right)^{\nu+2n}}{n!\Gamma(\nu+n+1)}, \quad z\in\mathbb{C}.
\end{equation*}

\begin{theorem}\label{thmriesz}   Let $x>0$ and $\rho>2\sigma_a^*-\delta-\frac12$.  Then the functional equation \eqref{18} is equivalent to the \emph{Riesz sum} identity
\begin{gather}
\df{1}{\Gamma(\rho+1)}{\sum_{\lambda_n\leq x}}^{\prime}a(n)(x-\lambda_n)^{\rho}=
\left(\df{1}{2\pi}\right)^{\rho}
\sum_{n=1}^{\infty}b(n)\left(\df{x}{\mu_n}\right)^{(\delta+\rho)/2}J_{\delta+\rho}(4\pi\sqrt{\mu_n x})+Q_{\rho}(x),\label{21}
\end{gather}
where the prime $\prime$ on the summation sign on the left side indicates that if $\rho=0$ and $x\in\{\lambda_n\}$, then only $\tf12 a(x)$ is counted.  Furthermore,
 $Q_{\rho}(x)$ is  defined by
\begin{equation}\label{22}
Q_{\rho}(x):=\df{1}{2\pi i}\int_{\mathcal{C}}\df{\chi(z)(2\pi)^zx^{z+\rho}}{\Gamma(\rho+1+z)}dz,
\end{equation}
where $\mathcal{C}$ is a curve or curves encircling  $\bf{S}$.
\end{theorem}

Chandrasekharan and Narasimhan \cite[p.~14, Theorem III]{cn1} show that the restriction $\rho>2\sigma_a^*-\delta-\frac12$ can be replaced by $\rho>2\sigma_a^*-\delta-\frac32$ under certain conditions. Because we later use analytic continuation, this extension is not important here.

Theorem \ref{thmmodular} is not explicitly used in the sequel.  However, Theorem \ref{thmriesz} is the key to our primary theorems, Theorem \ref{main} and Theorem \ref{maintheoremgeneral}.

Our examples include the following arithmetical functions: $r_k(n)$, the number of representations of $n$ as a sum of $k$ squares;  $\sigma_k(n)$, the sum of the $k$th powers of the divisors of $n$; Ramanujan's arithmetical function $\tau(n)$;  both odd and even primitive characters $\chi(n)$; and $F(n)$, the number of integral ideals of norm $n$ in an imaginary quadratic number field.

\section{Facts About Bessel Functions}

 The Bessel function of imaginary argument $I_{\nu}(z)$ is defined by \cite[p.~77]{watson}
\begin{equation}\label{b1}
I_{\nu}(z):=\sum_{n=0}^{\infty}\df{(\frac12 z)^{\nu+2n}}{n!\Gamma(\nu+n+1)},\quad z\in \mathbb{C},
\end{equation}
 while the modified Bessel function $K_{\nu}(z)$ is defined by \cite[p.~78]{watson}
\begin{align}\label{b2}
K_{\nu}(z)&:=\frac{\pi}{2}\df{I_{-\nu}(z)-I_{\nu}(z)}{\sin \nu\pi}, \quad z\in \mathbb{C}, \nu\notin\mathbb{Z},\\
K_n(z)&:=\lim_{\nu\to n}K_{\nu}(z), \qquad\qquad n\in\mathbb{Z}.\notag
\end{align}

As special cases \cite[p.~80]{watson},
\begin{align}
I_{1/2}\left(z\right)&=\sqrt{\frac{2}{\pi z}}\sinh  z,\label{I1}\\
K_{1/2}\left(z\right)&=\sqrt{\frac{\pi}{2 z}}e^{-z}.\label{I2}
\end{align}

 For $\nu\in\mathbb{C}$ \cite[p.~79]{watson},
\begin{equation}\label{KK}
K_{\nu}(z)=K_{-\nu}(z).
\end{equation}

For $\Re(\nu)>0$ \cite[p.~329]{bdkz},
\begin{equation}\label{limit}
\lim_{z\to 0}z^{\nu}K_{\nu}(z)=2^{\nu-1}\Gamma(\nu).
\end{equation}

The three foregoing Bessel functions satisfy the differentiation formulas \cite[pp.~66, 79]{watson}
\begin{align}
\df{d}{dz}\left(z^{\nu}J_{\nu}(z)\right)=&z^{\nu}J_{\nu-1}(z), \label{b3a}\\
\df{d}{dz}\left(z^{\nu}I_{\nu}(z)\right)=&z^{\nu}I_{\nu-1}(z), \label{b3}\\
\df{d}{dz}\left(z^{\nu}K_{\nu}(z)\right)=&-z^{\nu}K_{\nu-1}(z). \label{b4}
\end{align}

We shall need their asymptotic formulas as $z\to\infty$, namely \cite[pp.~199, 203, 202]{watson},
\begin{align}
J_{\nu}(z)=&\sqrt{\df{2}{\pi z}}\left(\cos(z-\tfrac12 \nu\pi-\tfrac14 \pi)+O\left(\df{1}{z}\right)\right),\label{b5a}\\
I_{\nu}(z)=&\sqrt{\df{1}{2\pi z}}e^z\left(1+O\left(\df{1}{z}\right)\right),\label{b5}\\
K_{\nu}(z)=&\sqrt{\df{\pi}{2 z}}e^{-z}\left(1+O\left(\df{1}{z}\right)\right).\label{b6}
\end{align}

\begin{lemma}\label{Watson1} \cite[p.~417]{watson} Let $a>0$, $\Re(\mu)>-1$, and $\nu\in\mathbb{C}$.  Then,
\begin{equation*}
\int_0^{\infty}\df{K_{\nu}(a\{t^2+z^2\})}{(t^2+z^2)^{\nu/2}}t^{2\mu+1}dt
=\df{2^{\mu}\Gamma(\mu+1)}{a^{\mu+1}z^{\nu-\mu-1}}K_{\nu-\mu-1}(az).
\end{equation*}
\end{lemma}

\begin{lemma}\label{lemma1} \cite[p.~416]{watson} For $a,b>0$, $\textup{Re}(z)>0, \textup{Re}(\mu)>-1$, and $\nu\in\mathbb{C}$,
\begin{equation*}
\int_0^{\infty}J_{\mu}(bx)\df{K_{\nu}(a\sqrt{z^2+x^2})}{(z^2+x^2)^{\nu/2}}x^{\mu+1}dx
=\df{b^{\mu}}{a^{\nu}}\left(\df{\sqrt{a^2+b^2}}{z}\right)^{\nu-\mu-1}K_{\nu-\mu-1}(z\sqrt{a^2+b^2}).
\end{equation*}
\end{lemma}

\section{The First Primary Theorem}

\begin{theorem}\label{main} Let $\Re(\nu)>-1$, $\Re(c), \Re(r)>0$, and $\rho>-1$.   Then,
\begin{gather}
\df{1}{\Gamma(\rho+1)}\sum_{n=1}^{\infty}a(n)\int_{\lambda_n}^{\infty}(x-\lambda_n)^{\rho}(c^2+x)^{-\nu/2}K_{\nu}(4\pi r\sqrt{c^2+x})dx\notag\\
=\df{1}{(2\pi)^{\rho+1}r^{\nu}c^{\nu-\delta-\rho-1}}
\sum_{n=1}^{\infty}\df{b(n)}{(r^2+\mu_n)^{(\delta+\rho-\nu+1)/2}}K_{\delta+\rho+1-\nu}(4\pi c\sqrt{r^2+\mu_n})\notag\\
+\int_0^{\infty}Q_{\rho}(x)(c^2+x)^{-\nu/2}K_{\nu}(4\pi r\sqrt{c^2+x})dx,\label{1a}
\end{gather}
where it is assumed that the integral $Q_{\rho}(x)$, defined by \eqref{22}, converges absolutely.
\end{theorem}

\begin{proof} Assume that $\rho>2\sigma_a^*-\delta-\frac12$.  Multiply both sides of \eqref{21} by
\begin{equation*}
(c^2+x)^{-\nu/2}K_{\nu}(4\pi r\sqrt{c^2+x}), \qquad c,r>0,
\end{equation*}
  and integrate over $0\leq x<\infty$.  Let $F_1(\delta,\rho,\nu)$ denote the left-hand side and let $F_2(\delta,\rho,\nu)$ and $F_3(\delta,\rho,\nu)$ denote, in order, the two terms on the right-hand side that we so obtain.

  First,
  \begin{align}\label{1}
  F_1(\delta,\rho,\nu)=&\df{1}{\Gamma(\rho+1)}\int_0^{\infty}{\sum_{\lambda_n\leq x}}^{\prime}a(n)(x-\lambda_n)^{\rho}(c^2+x)^{-\nu/2}K_{\nu}(4\pi r\sqrt{c^2+x})dx\notag\\
  =&\df{1}{\Gamma(\rho+1)}
  \sum_{n=1}^{\infty}a(n)\int_{\lambda_n}^{\infty}(x-\lambda_n)^{\rho}(c^2+x)^{-\nu/2}K_{\nu}(4\pi r\sqrt{c^2+x})dx.
  \end{align}

  Second, after we invert the order of summation and integration by absolute convergence on the right-hand side,
  we are led to the integral
\begin{align}\label{2}
I(\delta,\rho,\nu):=&\int_0^{\infty}x^{(\delta+\rho)/2}
(c^2+x)^{-\nu/2}J_{\delta+\rho}(4\pi\sqrt{\mu_n x})K_{\nu}(4\pi r\sqrt{c^2+x})dx\notag\\
=&2\int_0^{\infty}u^{\delta+\rho+1}J_{\delta+\rho}(4\pi u\sqrt{\mu_n})
\df{K_{\nu}(4\pi r\sqrt{c^2+u^2})}{(c^2+u^2)^{\nu/2}}du.
\end{align}
Apply Lemma \ref{lemma1} with
$$ \mu=\delta+\rho,\qquad a=4\pi r, \qquad \text{and} \qquad b=4\pi\sqrt{\mu_n}.$$
Hence, from \eqref{2}, for $\delta+\rho>-1$,
\begin{align}\label{4}
I(\delta,\rho,\nu)=&
2\df{(4\pi\sqrt{\mu_n})^{\delta+\rho}}{(4\pi r)^{\nu}}
\left(\df{\sqrt{(4\pi r)^2+(4\pi\sqrt{\mu_n})^2}}{c}\right)^{\nu-\delta-\rho-1}\notag\\
&\times K_{\nu-\delta-\rho-1}\left(c\sqrt{(4\pi r)^2+(4\pi\sqrt{\mu_n})^2}\right)\notag\\
=&\df{\mu_n^{(\delta+\rho)/2}}{2\pi r^{\nu}c^{\nu-\delta-\rho-1}}(r^2+\mu_n)^{(\nu-\delta-\rho-1)/2}
K_{\nu-\delta-\rho-1}(4\pi c\sqrt{r^2+\mu_n}).
\end{align}
In summary, with the use of \eqref{4} and \eqref{KK}, we have
\begin{align}\label{5}
F_2(\delta,\rho,\nu)=&\df{1}{(2\pi)^{\rho+1}r^{\nu}c^{\nu-\delta-\rho-1}}
\sum_{n=1}^{\infty}\df{b(n)}{(r^2+\mu_n)^{(\delta+\rho-\nu+1)/2}}K_{\nu-\delta-\rho-1}(4\pi c\sqrt{r^2+\mu_n})\notag\\
=&\df{1}{(2\pi)^{\rho+1}r^{\nu}c^{\nu-\delta-\rho-1}}
\sum_{n=1}^{\infty}\df{b(n)}{(r^2+\mu_n)^{(\delta+\rho-\nu+1)/2}}K_{\delta+\rho+1-\nu}(4\pi c\sqrt{r^2+\mu_n}).
\end{align}

Thirdly,
\begin{align}\label{6}
F_3(\delta,\rho,\nu)=&\int_0^{\infty}Q_{\rho}(x)(c^2+x)^{-\nu/2}K_{\nu}(4\pi r\sqrt{c^2+x})dx.
\end{align}

We now gather together \eqref{1}, \eqref{5}, and \eqref{6} to conclude \eqref{1a}, which we have proved for $\nu, r, c>0$.  However, in view of \eqref{b6}, we see that by analytic continuation, \eqref{1a} holds for
$\Re(\nu)>-1$, and $\Re(c), \Re(r)>0$.  The conditions $\rho>2\sigma_a^*-\delta-\frac12$ and $\delta+\rho>-1$ can be discarded by analytic continuation in $\rho$.
\end{proof}

\section{The Special Case $\rho=0$}
We consider Theorem \ref{main} in the special case $\rho=0$.

\begin{theorem}\label{maintheorem} Let $\Re(\nu)>-1$ and $\Re(c), \Re(r)>0$. Assume that the integral below converges absolutely.  Then,
\begin{gather*}
\df{1}{2\pi r}\sum_{n=1}^{\infty}\df{a(n)}{(c^2+\lambda_n)^{(\nu-1)/2}}
K_{\nu-1}(4\pi r\sqrt{c^2+\lambda_n})\notag\\
=\df{1}{2\pi r^{\nu}c^{\nu-\delta-1}}
\sum_{n=1}^{\infty}\df{b(n)}{(r^2+\mu_n)^{(\delta-\nu+1)/2}}K_{\delta+1-\nu}(4\pi c\sqrt{r^2+\mu_n})\notag\\
+\int_0^{\infty}Q_{0}(x)(c^2+x)^{-\nu/2}K_{\nu}(4\pi r\sqrt{c^2+x})dx.
\end{gather*}
\end{theorem}

\begin{proof} First, set
$$ u=4\pi r\sqrt{c^2+x} \quad \Rightarrow dx=\df{u}{8\pi^2r^2}du.$$
Hence, in turn, using \eqref{KK}, \eqref{b4}, and \eqref{KK}, we find that
\begin{align*}
\int_{\lambda_n}^{\infty}(c^2+x)^{-\nu/2}&K_{\nu}(4\pi r\sqrt{c^2+x})dx
=\int_{4\pi r\sqrt{c^2+\lambda_n}}^{\infty}\left(\df{u}{4\pi r}\right)^{-\nu}\left(\df{u}{8\pi^2r^2}\right)K_{\nu}(u)du\notag\\
=&\df{2}{(4\pi r)^{2-\nu}}\int_{4\pi r\sqrt{c^2+\lambda_n}}^{\infty}u^{-\nu+1}K_{\nu}(u)du\notag\\
=&2(4\pi r)^{\nu-2}\int_{4\pi r\sqrt{c^2+\lambda_n}}^{\infty}u^{-\nu+1}K_{-\nu}(u)du\notag\\
=&-2(4\pi r)^{\nu-2}\int_{4\pi r\sqrt{c^2+\lambda_n}}^{\infty}\df{d}{du}\left(u^{-\nu+1}K_{-\nu+1}(u)\right)du\notag\\
=&2(4\pi r)^{\nu-2}(4\pi r\sqrt{c^2+\lambda_n})^{-\nu+1}K_{-\nu+1}(4\pi r\sqrt{c^2+\lambda_n})\notag\\
=&\df{1}{2\pi r}(c^2+\lambda_n)^{-(\nu-1)/2}K_{\nu-1}(4\pi r\sqrt{c^2+\lambda_n}).
\end{align*}
Thus,  the sum on the left-hand side of \eqref{1a} reduces to
\begin{equation*}
\df{1}{2\pi r}\sum_{n=1}^{\infty}\df{a(n)}{(c^2+\lambda_n)^{(\nu-1)/2}}
K_{\nu-1}(4\pi r\sqrt{c^2+\lambda_n}).
\end{equation*}
The remaining part of the proof is immediate after setting $\rho=0$ in Theorem \ref{main}.
\end{proof}

Before giving examples in illustration of Theorem  \ref{maintheorem}, we offer remarks on previous work.   Theorem \ref{main} is new.    The first author's paper \cite[p.~342]{III} contains the first statement and proof of Theorem \ref{maintheorem} \cite[pp.~342--344]{III}. Our proof here is completely different from that in \cite{III}.   Theorem \ref{maintheorem} was also established via the Vorono\"{\i} summation formula in \cite[p.~154]{V}.  The special case, $\delta=1$, of Theorem \ref{maintheorem} was first established by F.~Oberhettinger and K.~Soni \cite[p.~24]{os} in 1972.

To illuminate the equivalence of the functional equation \eqref{19}, the modular relation \eqref{modular}, and the Riesz sum identity \eqref{21}, Chandrasekharan and Narasimhan \cite{cn1} examine the three identities with particular arithmetical functions. For more details about the functional equations associated with these arithmetical functions, and for calculations of $Q_0(x)$, see their paper \cite{cn1}.

In the examples below we refer to calculations made  by Chandrasekharan and Narasimhan \cite{cn1} to illustrate Theorem \ref{thmriesz}.  In particular, we use a few of their determinations of $Q_{\rho}(x)$.

\section{Example: $r_k(n)$}

 Let $r_k(n)$ denote the number of representations of the positive integer $n$ as a sum of $k$ squares, where  $k\geq2$.  Then
$$\zeta_k(s):=\sum_{n=1}^{\infty}\df{r_k(n)}{n^s},\qquad \sigma>k/2,$$
satisfies the functional equation
\begin{equation}\label{functionalequation}
 \pi^{-s}\Gamma(s)\zeta_k(s)=\pi^{s-k/2}\Gamma(k/2-s)\zeta_k(k/2-s).
 \end{equation}
In the notation of \eqref{19},
\begin{equation*}
a(n)=b(n)=r_k(n), \qquad \delta=\df{k}{2}, \qquad \text{and} \qquad \lambda_n=\mu_n=\df{n}{2}.
\end{equation*}
 From the functional equation \eqref{functionalequation}, $\zeta_k(0)=-1$, and $\zeta_k(s)$ has a simple pole at $s=2k$ with residue $\pi^{k/2}/\Gamma(k/2)$.
 It readily follows that
 \begin{equation}\label{Q}
Q_{0}(x)=-1+\df{(2\pi x)^{k/2}}{\Gamma(1+k/2)}.
\end{equation}

Appealing to Theorem \ref{maintheorem} and \eqref{Q}, we find that
\begin{align}\label{watson9}
&\df{1}{2\pi r}\sum_{n=1}^{\infty}\df{r_k(n)}{(c^2+n/2)^{(\nu-1)/2}}K_{\nu-1}(4\pi r\sqrt{c^2+n/2})\notag\\
=&\df{1}{2\pi r^{\nu}c^{\nu-k/2-1}}\sum_{n=1}^{\infty}\df{r_k(n)}{(r^2+n/2)^{(k/2-\nu+1)/2}}
K_{k/2+1-\nu}(4\pi c\sqrt{r^2+n/2})\notag\\
&+\int_0^{\infty}\left(-1+\df{(2\pi x)^{k/2}}{\Gamma(1+k/2)}\right)(c^2+x)^{-\nu/2}
K_{\nu}(4\pi r\sqrt{c^2+x})dx.
\end{align}

First, making the trivial change of variable $x=t^2$, and applying Lemma \ref{Watson1} with $ a=4\pi r,\,\,z=c,$ and $\mu=0$, we find that
\begin{equation}\label{watson10}
-\int_0^{\infty}(c^2+x)^{-\nu/2}
K_{\nu}(4\pi r\sqrt{c^2+x})dx=-\df{1}{2\pi rc^{\nu-1}}K_{\nu-1}(4\pi rc).
\end{equation}
Second, again making the trivial change of variable $x=t^2$, and applying Lemma \ref{Watson1} with $ a=4\pi r, z=c,$ and $\mu=\frac12 k$, we find that
\begin{equation}\label{watson11}
\int_0^{\infty}\df{(2\pi x)^{k/2}}{\Gamma(1+k/2)}(c^2+x)^{-\nu/2}
K_{\nu}(4\pi r\sqrt{c^2+x})dx
=\df{1}{2\pi  r^{k/2+1}c^{\nu-k/2-1}}K_{\nu-k/2-1}(4\pi rc).
\end{equation}
Now, put \eqref{watson10} and \eqref{watson11} into \eqref{watson9}. If we define $r_k(0)=1$ and use \eqref{limit}, we see that \eqref{watson10} can be written as the term for $n=0$ in the series on the left-hand side of \eqref{watson9}, while \eqref{watson11} can be considered as the term for $n=0$ in the series on the right-hand side.
Multiplying both sides of the resulting identity by $2\pi r$, and replacing $\nu$ by $\nu+1$, we conclude that
\begin{gather}
\sum_{n=0}^{\infty}\df{r_k(n)}{(c^2+n/2)^{\nu/2}}K_{\nu}(4\pi r\sqrt{c^2+n/2})\notag\\=
\df{1}{r^{\nu}c^{\nu-k/2}}\sum_{n=0}^{\infty}\df{r_k(n)}{(r^2+n/2)^{(k/2-\nu)/2}}K_{k/2-\nu}(4\pi c\sqrt{r^2+n/2}).
\label{watson4}
\end{gather}

The identity \eqref{watson4} was also established by the first author, Y.~Lee, and J.~Sohn \cite[p.~39, Equation (5.5)]{bls}.  For $k=2$, \eqref{watson4} was first proved by A.~L.~Dixon and W.~L.~Ferrar \cite[p.~53, Equation (4.13)]{dixonferrar} in 1934. A different proof for $k=2$ was given by Oberhettinger and Soni \cite[p.~24]{os}.

\section{Example: $\sigma_k(n)$}

Let $\sigma_k(n)$ denote the sum of the $k$th powers of the divisors of $n$, where it is assumed that $k$ is an odd positive integer.  The generating function for $\sigma_k(n)$ is given by
\begin{equation*}
\zeta_k(s):=\zeta(s)\zeta(s-k)=\sum_{n=1}^{\infty}\df{\sigma_k(n)}{n^s}, \qquad \sigma > k+1,
\end{equation*}
and it satisfies the functional equation
\begin{equation}\label{41a}
(2\pi)^{-s}\Gamma(s)\zeta_k(s)=(-1)^{(k+1)/2}(2\pi)^{-(k+1-s)}\Gamma(k+1-s)\zeta_k(k+1-s).
\end{equation}
In the notation of the  Dirichlet series and functional equation in \eqref{18} and \eqref{19}, respectively,
\begin{equation*}
   a(n)=\sigma_k(n), \quad b(n)=(-1)^{(k+1)/2}\sigma_k(n), \quad\lambda_n=\mu_n=n,\quad \delta=k+1.
   \end{equation*}

Now $Q_0(s)$ is the sum of the residues of
\begin{equation*}
R(z):=\df{\Gamma(z)\zeta(z)\zeta(z-k)x^{z}}{\Gamma(z+1)}.
\end{equation*}
(In Chandrasekaran and Narasimhan's paper \cite{cn1}, they utilize a different convention for Bernoulli numbers, and so our representation for $Q_0$ takes a different form from theirs.)
Observe that $R(z)$ has simple poles at $z=0,-1,k+1$. Using Euler's formula,
$$\zeta(2n)=(-1)^{n-1}\df{(2\pi)^{2n}B_{2n}}{2(2n)!}, \qquad n\geq 1.$$
where $n$ is a positive integer and $B_{n}$ denotes the $n$th Bernoulli number, we readily find that
\begin{equation}\label{45a}
Q_0(x)=\df{B_{k+1}}{2(k+1)}-\df{\delta_{1,k}x}{2}
+\df{(2\pi)^{k+1}(-1)^{(k-1)/2}B_{k+1}x^{k+1}}{2(k+1)\Gamma(k+2)},
\end{equation}
  where
  \begin{equation*}
  \delta_{1,k}=\begin{cases}
  1, \quad \text{ if } k=1,\\
  0, \quad \text{otherwise}.
  \end{cases}
  \end{equation*}

Applying Theorem \ref{maintheorem} and employing \eqref{45a}, we find that
  \begin{align}\label{watson12}
  &\df{1}{2\pi r}\sum_{n=1}^{\infty}\df{\sigma_k(n)}{(c^2+n)^{(\nu-1)/2}}K_{\nu-1}(4\pi r\sqrt{c^2+n})\notag\\
=&\df{1}{2\pi r^{\nu}c^{\nu-k-2}}\sum_{n=1}^{\infty}\df{(-1)^{(k+1)/2}\sigma_k(n)}{(r^2+n)^{(k+2-\nu)/2}}
K_{k+2-\nu}(4\pi c\sqrt{r^2+n})\notag\\
&+\int_0^{\infty}\left(\df{B_{k+1}}{2(k+1)}-\df{\delta_{1,k}x}{2}
+\df{(2\pi)^{k+1}(-1)^{(k-1)/2}B_{k+1}x^{k+1}}{2(k+1)\Gamma(k+2)}\right)(c^2+x)^{-\nu/2}
K_{\nu}(4\pi r\sqrt{c^2+x})dx.
\end{align}

Let $I_1, I_2$, and $I_3$ denote, respectively, the three integrals on the right side of \eqref{watson12}.  In each instance below, we initially make the change of variable $x=t^2$. First, by Lemma \ref{Watson1}, as in the calculation of \eqref{watson10},
\begin{equation}\label{watson13}
I_1=\df{B_{k+1}}{2(k+1)}\int_0^{\infty}(c^2+x)^{-\nu/2}K_{\nu}(4\pi r\sqrt{c^2+x})dx
=\df{B_{k+1}}{k+1}\,\df{1}{4\pi rc^{\nu-1}}K_{\nu-1}(4\pi rc).
\end{equation}
Secondly, apply Lemma \ref{Watson1} with $a=4\pi r, z=c$, and $\mu=1$. Hence,
\begin{equation}\label{watson14}
I_2=-\df{\delta_{1,k}}{2}\int_0^{\infty}(c^2+x)^{-\nu/2}K_{\nu}(4\pi r\sqrt{c^2+x})x\,dx
=-\delta_{1,k}\df{2}{(4\pi r)^2c^{\nu-2}}K_{\nu-2}(4\pi rc).
\end{equation}
Thirdly, we apply Lemma \ref{Watson1} with $a=2\pi r, z=c$, and $\mu=k+1$. Therefore,
\begin{align}\label{watson15}
I_3=&\df{(2\pi)^{k+1}(-1)^{(k-1)/2}B_{k+1}}{2(k+1)\Gamma(k+2)}
\int_0^{\infty}(c^2+x)^{-\nu/2}K_{\nu}(4\pi r\sqrt{c^2+x})x^{k+1}\,dx\notag\\
=&\df{(-1)^{(k-1)/2}B_{k+1}}{4\pi(k+1)r^{k+2}c^{\nu-k-2}}K_{\nu-k-2}(4\pi rc).
\end{align}

In summary, putting \eqref{watson13}--\eqref{watson15} into \eqref{watson12}, we deduce that
\begin{align}\label{watson16}
&\df{1}{2\pi r}\sum_{n=1}^{\infty}\df{\sigma_k(n)}{(c^2+n)^{(\nu-1)/2}}K_{\nu-1}(4\pi r\sqrt{c^2+n})\notag\\
=&\df{1}{2\pi r^{\nu}c^{\nu-k-2}}\sum_{n=1}^{\infty}\df{(-1)^{(k+1)/2}\sigma_k(n)}{(r^2+n)^{(k+2-\nu)/2}}
K_{k+2-\nu}(4\pi c\sqrt{r^2+n})
+\df{B_{k+1}}{k+1}\,\df{1}{4\pi rc^{\nu-1}}K_{\nu-1}(4\pi rc)\notag\\
&-\delta_{1,k}\df{2}{(4\pi r)^2c^{\nu-2}}K_{\nu-2}(4\pi rc)
+\df{(-1)^{(k-1)/2}B_{k+1}}{4\pi(k+1)r^{k+2}c^{\nu-k-2}}K_{\nu-k-2}(4\pi rc).
\end{align}

We now put \eqref{watson16} in a more palatable form.  From \eqref{41a},
$$\zeta_k(0)=\zeta(0)\zeta(-k)=-\df12\cdot-\df{B_{k+1}}{k+1}=\df{B_{k+1}}{2(k+1)},$$
by \cite[p.~12]{edwards}.
Define
\begin{equation}\label{49cc}
\sigma_k(0)=-\zeta_k(0)=-\df{B_{k+1}}{2(k+1)}.
\end{equation}
Thus, by \eqref{49cc}, the first expression after the series on the right-hand side of \eqref{watson16} can be expressed as the term for $n=0$ in the series on the left-hand side.  Similarly, the last expression on the right-hand side of \eqref{watson16} can be represented as the term for $n=0$ in the series on the right-hand side of \eqref{watson16}.  Thus, we can write \eqref{watson16} in the more simplified form
\begin{gather}
\df{1}{2\pi r}\sum_{n=0}^{\infty}\df{\sigma_k(n)}{(c^2+n)^{(\nu-1)/2}}K_{\nu-1}(4\pi r\sqrt{c^2+n})\notag\\
=\df{1}{2\pi r^{\nu}c^{\nu-k-2}}\sum_{n=0}^{\infty}\df{(-1)^{(k+1)/2}\sigma_k(n)}{(r^2+n)^{(k+2-\nu)/2}}
K_{k+2-\nu}(4\pi c\sqrt{r^2+n})
-\delta_{1,k}\df{2}{(4\pi r)^2c^{\nu-2}}K_{\nu-2}(4\pi rc).\label{watson16+}
\end{gather}
 Multiplying both sides of \eqref{watson16+} by $2\pi r$, and replacing $\nu$ by $\nu+1$, we deduce that
\begin{gather}
\sum_{n=0}^{\infty}\df{\sigma_k(n)}{(c^2+n)^{\nu/2}}K_{\nu}(4\pi r\sqrt{c^2+n})\notag\\
=-\df{\delta_{1,k}}{4\pi rc^{\nu-1}}K_{\nu-1}(4\pi rc)
+\df{1}{r^{\nu}c^{\nu-k-1}}\sum_{n=0}^{\infty}\df{(-1)^{(k+1)/2}\sigma_k(n)}{(r^2+n)^{(k+1-\nu)/2}}\cdot
K_{k+1-\nu}(4\pi c\sqrt{r^2+n}).\label{watson5}
\end{gather}
This identity appears to be new.

In a 3-page fragment published with his lost notebook \cite[p.~253]{lnb}, Ramanujan offered a kindred formula to \eqref{watson5}.
If $\alpha$ and $\beta$ are positive numbers such that $\alpha\beta=\pi^2$, and if $s$ is any
complex number, then
\begin{align}\label{253}
&\sqrt{\alpha}\sum_{n=1}^{\infty}\sigma_{-s}(n)n^{s/2}K_{s/2}(2n\alpha)
-\sqrt{\beta}\sum_{n=1}^{\infty}\sigma_{-s}(n)n^{s/2}K_{s/2}(2n\beta)\notag\\
=&\df14\Gamma\left(\df{s}{2}\right)\zeta(s)\{\beta^{(1-s)/2}-\alpha^{(1-s)/2}\}
+\df14\Gamma\left(-\df{s}{2}\right)\zeta(-s)\{\beta^{(1+s)/2}-\alpha^{(1+s)/2}\}.
\end{align}
 Note that \eqref{253} is \emph{not} a special case of \eqref{watson5}, and also note that \eqref{253} is valid for \emph{all} complex $s$, while $k$ in \eqref{watson5} is an odd positive integer.

Unaware that \eqref{253} was first established by Ramanujan \cite[p.~253]{lnb}, A.~P.~Guinand gave the first proof in print in 1955 \cite{guinand}.    The identity \eqref{253} is now known as Guinand's formula or the Ramanujan--Guinand formula.  See also \cite[pp.~25--27]{bls} for a proof.
Letting $s=0$ in \eqref{253}, we obtain a well-known formula of Koshliakov \cite{bls}.

\section{Example: $\tau(n)$}

Recall that the Dirichlet series for Ramanujan's arithmetical function $\tau(n)$
\begin{equation}\label{50a}
f(s):=\sum_{n=1}^{\infty}\df{\tau(n)}{n^s}, \quad \sigma > \df{13}{2},
\end{equation}
satisfies the functional equation
\begin{equation}\label{50}
\chi(s):=(2\pi)^{-s}\Gamma(s)f(s)=(2\pi)^{-(12-s)}\Gamma(12-s)f(12-s).
\end{equation}
The function $\chi(s)$ is an entire function, and so $Q_0(x)\equiv0$.  Clearly,
\begin{equation}\label{exam4}
\lambda_n=\mu_n=n \qquad\delta=12.
\end{equation}
Applying Theorem \ref{maintheorem} and replacing $\nu$ by $\nu+1$, we deduce that, for $\Re(\nu), \Re(c), \Re(r) >0$,
\begin{equation}\label{watson3}
\sum_{n=1}^{\infty}\df{\tau(n)}{(c^2+n)^{\nu/2}}K_{\nu}(4\pi r\sqrt{c^2+n})=
\df{1}{r^{\nu}c^{\nu-12}}\sum_{n=1}^{\infty}\df{\tau(n)}{(r^2+n)^{(12-\nu)/2}}K_{12-\nu}(4\pi c\sqrt{r^2+n}).
\end{equation}

The identity \eqref{watson3} was first established by the first author, Lee, and Sohn \cite[p.~40, Equation (5.7)]{bls}.

\section{Example: Primitive characters $\chi(n)$}
Let $\chi$ denote a primitive character modulo $q$.  Because the functional equations for the Dirichlet $L$-series
$$L(s,\chi)=\sum_{n=1}^{\infty}\df{\chi(n)}{n^s}, \qquad \sigma >0,$$
are different for $\chi$ even and $\chi$ odd, we separate the two cases.

Suppose first that $\chi$ is odd.  Then the functional equation for $L(s,\chi)$ is given by \cite[p.~71]{davenport}
\begin{equation}\label{funcequa1}
\chi(s):=\left(\df{\pi}{q}\right)^{-s}\Gamma(s)L(2s-1,\chi)=
-\df{i\tau(\chi)}{\sqrt{q}}\left(\df{\pi}{q}\right)^{-(\tf32-s)}\Gamma\left(\tf32-s\right)L(2-2s,\overline{\chi}),
\end{equation}
where $\overline{\chi}(n)$ denotes the complex conjugate of $\chi(n)$, and $\tau(\chi)$ denotes the Gauss sum
\begin{equation}\label{gauss}
\tau(\chi):=\sum_{n=1}^{q}\chi(n)e^{2\pi in/q}.
\end{equation}
Hence, in the notation of \eqref{18} and \eqref{19},
\begin{equation}\label{exam5}
a(n)=n\chi(n),\quad b(n)=-\df{i\tau(\chi)}{\sqrt{q}}n\overline{\chi}(n), \quad \lambda_n=\mu_n=\df{n^2}{2q},
\quad \delta=\df32.
\end{equation}
 Also, $\chi(s)$ is an entire function, and consequently $Q_0(x)\equiv0$.
 Applying Theorem \ref{maintheorem}, multiplying both sides of the resulting identity by $2\pi r$, and replacing $\nu$ by $\nu+1$,  we conclude that
 \begin{gather*}
\sum_{n=0}^{\infty}\df{n\chi(n)}{(c^2+n^2/(2q))^{\nu/2}}K_{\nu}(4\pi r\sqrt{c^2+n^2/(2q)})\notag\\=-
\df{i\tau(\chi)}{r^{\nu}c^{\nu-3/2}\sqrt{q}}
\sum_{n=0}^{\infty}\df{n\,\overline{\chi}(n)}{(r^2+n^2/(2q))^{(3/2-\nu)/2}}K_{3/2-\nu}(4\pi c\sqrt{r^2+n^2/(2q)}),
\end{gather*}
which is new.

Second, let $\chi$ be even.  Then the functional equation of $L(s,\chi)$ is given by \cite[p.~69]{davenport}
 \begin{equation}\label{funcequa2}
\chi(s):=\left(\df{\pi}{q}\right)^{-s}\Gamma(s)L(2s,\chi)=
\df{\tau(\chi)}{\sqrt{q}}\left(\df{\pi}{q}\right)^{-(\tf12-s)}\Gamma\left(\tf12-s\right)L(1-2s,\overline{\chi}).
\end{equation}
Hence, by \eqref{18} and \eqref{19},
\begin{equation}\label{exam6}
a(n)=\chi(n),\quad b(n)=\df{\tau(\chi)}{\sqrt{q}}\overline{\chi}(n), \quad \lambda_n=\mu_n=\df{n^2}{2q},
\quad \delta=\df12.
\end{equation}
Also, $\chi(s)$ is an entire function, and consequently $Q_0(x)\equiv0$.
Appealing to Theorem \ref{maintheorem},  multiplying both sides of the identity so obtained by $2\pi r$, and replacing $\nu$ by $\nu+1$, we conclude that
\begin{gather*}
\sum_{n=0}^{\infty}\df{\chi(n)}{(c^2+n^2/(2q))^{\nu/2}}K_{\nu}(4\pi r\sqrt{c^2+n^2/(2q)})\notag\\=
\df{\tau(\chi)}{r^{\nu}c^{\nu-1/2}\sqrt{q}}
\sum_{n=0}^{\infty}\df{\overline{\chi}(n)}{(r^2+n^2/(2q))^{(1/2-\nu)/2}}K_{1/2-\nu}(4\pi c\sqrt{r^2+n^2/(2q)}),
\end{gather*}
which is also a new identity.

\section{Example: Ideal Functions $F(n)$ of Imaginary Quadratic Number Fields}

Let $F(n)$ denote the number of integral ideals of norm $n$ in an imaginary quadratic number  field $K=\mathbb{Q}\left(\sqrt{-D}\right)$, where $D$ is the discriminant of $K$.  Then the Dedekind zeta function
$$\zeta_{K}(s):=\sum_{n=1}^{\infty}\df{F(n)}{n^s}, \qquad \sigma>1,$$
satisfies the functional equation \cite[p.~211]{cohen}
\begin{equation}\label{dedekindk}
\left(\df{2\pi}{\sqrt{D}}\right)^{-s}\Gamma(s)\zeta_K(s)
=\left(\df{2\pi}{\sqrt{D}}\right)^{s-1}\Gamma(1-s)\zeta_K(1-s).
\end{equation}
We note from \eqref{18} and \eqref{19} that
\begin{equation*}
 a(n)=b(n)=F(n), \qquad \lambda_n=\mu_n=n/\sqrt{D}, \qquad \delta=1.
 \end{equation*}
The function $\zeta_K(s)$ has an analytic continuation to the entire complex plane where it is analytic except for a simple pole at $s=1$.  From \cite[p.~212]{cohen},
\begin{equation}\label{pole}
\lim_{s\to 1}(s-1)\zeta_K(s)=\df{2\pi h(K)R(K)}{w(K)\sqrt{D}},
\end{equation}
where $h(K), R(K)$, and $w(K)$ denote, respectively, the class number of $K$, the regulator of $K$, and the number of roots of unity in $K$.
Furthermore, from \eqref{dedekindk} and \eqref{pole},
\begin{equation}\label{zerovalue}
\zeta_K(0)=\lim_{s\to 0}\df{\sqrt{D}}{2\pi}\cdot\df{1}{s\Gamma(s)}\cdot s\zeta_K(1-s)
=\df{\sqrt{D}}{2\pi}\cdot-\df{2\pi h(K)R(K)}{w(K)\sqrt{D}}=-\df{h(K)R(K)}{w(K)}.
\end{equation}
For simplicity, set $d=\sqrt{D},  h=h(K), R=R(K)$, and $w=w(K)$.
From \eqref{zerovalue} and \eqref{pole},
\begin{align}
Q_0(x)=&\df{1}{2\pi i}\int_{\mathcal{C}}\df{\Gamma(z)}{\Gamma(z+1)}d^z
\zeta_{K}(z)x^zdz
=-\df{hR}{w}+\df{2\pi hRx}{w}.\label{70}
\end{align}

 By Theorem \ref{maintheorem} and \eqref{70},
\begin{align}\label{watson18}
&\df{1}{2\pi r}\sum_{n=1}^{\infty}\df{F(n)}{(c^2+n/d)^{(\nu-1)/2}}K_{\nu-1}(4\pi r\sqrt{c^2+n/d})\notag\\
=&\df{1}{2\pi r^{\nu}c^{\nu-2}}\sum_{n=1}^{\infty}\df{F(n)}{(r^2+n/d)^{(2-\nu)/2}}
K_{2-\nu}(4\pi c\sqrt{r^2+n/d})\notag\\
&+\int_0^{\infty}\left(-\df{hR}{w}+\df{2\pi hRx}{w}\right)(c^2+x)^{-\nu/2}K_{\nu}(4\pi r\sqrt{c^2+x})dx.
\end{align}
Let $I_1$ and $I_2$ denote, respectively, the two integrals on the right-hand side of \eqref{watson18}.  First, by Lemma \ref{Watson1}, as we did in our calculations in \eqref{watson10} and \eqref{watson13},
\begin{equation}\label{watson19}
I_1=-\df{hR}{w}\int_0^{\infty}(c^2+x)^{-\nu/2}K_{\nu}(4\pi r\sqrt{c^2+x})dx=
-\df{hR}{w}\df{1}{2\pi rc^{\nu-1}}K_{\nu-1}(4\pi rc).
\end{equation}
Second, by Lemma \ref{Watson1} with the same calculation as in \eqref{watson14},
\begin{equation}\label{watson20}
I_2=\df{2\pi hR}{w}\int_0^{\infty}x(c^2+x)^{-\nu/2}K_{\nu}(4\pi r\sqrt{c^2+x})dx
=\df{hR}{2w\,\pi r^2c^{\nu-2}}K_{\nu-2}(4\pi rc).
\end{equation}
Suppose that we define
\begin{equation}\label{hR}
F(0)=\df{hR}{w}.
\end{equation}
Then, substituting \eqref{watson19} and \eqref{watson20} into \eqref{watson18} and employing the definition \eqref{hR}  to identify \eqref{watson19} and \eqref{watson20} as the terms for $n=0$ on the left- and right-hand  sides below, we find that
\begin{align}\label{watson22}
&\df{1}{2\pi r}\sum_{n=0}^{\infty}\df{F(n)}{(c^2+n/d)^{(\nu-1)/2}}K_{\nu-1}(4\pi r\sqrt{c^2+n/d})\notag\\
=&\df{1}{2\pi r^{\nu}c^{\nu-2}}\sum_{n=0}^{\infty}\df{F(n)}{(r^2+n/d)^{(2-\nu)/2}}
K_{2-\nu}(4\pi c\sqrt{r^2+n/d}).
\end{align}
Lastly, multiplying both sides of \eqref{watson22} by $2\pi r$ and replacing $\nu$ by $\nu+1$, we conclude with the identity
\begin{equation}\label{F}
\sum_{n=0}^{\infty}\df{F(n)}{(c^2+n/d)^{\nu/2}}K_{\nu}(4\pi r\sqrt{c^2+n/d})=
\df{1}{r^{\nu}c^{\nu-1}}\sum_{n=0}^{\infty}\df{F(n)}{(r^2+n/d)^{(1-\nu)/2}}K_{1-\nu}(4\pi c\sqrt{r^2+n/d}).
\end{equation}

The identity \eqref{F} was first proved in 1934 by N.~S.~Koshliakov \cite[p.~555, Equation (15)]{kosh2}, who used the Abel-Plana summation formula.

\section{The Second Primary Theorem}

\begin{theorem}\label{maintheoremgeneral}
For $\Re(\nu)>-1$, $\rho>-1$, $\delta+\rho+\Re(\nu)+1>\delta_a^*>0$, and $\Re(\sqrt{\a})>\Re(\sqrt{\b})>0$,
\begin{align}
&\frac{1}{\Gamma(\rho+1)}\sum_{n=1}^{\infty}a(n)\int_{\l_n}^{\infty}(t-\l_n)^\rho\frac{d}{dt}\left\{I_{\nu+1}\left(\pi \sqrt{t}\left(\sqrt{\a}-\sqrt{\b}\right)\right)K_{\nu+1}\left(\pi \sqrt{t}\left(\sqrt{\a}+\sqrt{\b}\right)\right) \right\}dt\nonumber\\
&=-\frac{2}{(2\pi)^{\d+2\rho}}\frac{\Gamma(\nu+\delta+\rho+1)}{\Gamma(\nu+2)}
\sum_{n=1}^{\infty}\frac{b(n)}{\sqrt{4\mu_n+\a}\sqrt{4\mu_n+\b}}
\left(\frac{\sqrt{4\mu_n+\a}-\sqrt{4\mu_n+\b}}{\sqrt{4\mu_n+\a}+\sqrt{4\mu_n+\b}} \right)^{\nu+1}\nonumber\\
&\quad\times\left(\frac{1}{\sqrt{4\mu_n+\a}} +\frac{1}{\sqrt{4\mu_n+\b}}\right)^{2\delta+2\rho-2}
{}_2F_{1}\left[\nu-\delta-\rho+2,1-\delta-\rho;\nu+2;
\left(\frac{\sqrt{4\mu_n+\a}-\sqrt{4\mu_n+\b}}{\sqrt{4\mu_n+\a}+\sqrt{4\mu_n+\b}} \right)^{2} \right]\nonumber\\
&\quad-\frac{Q_{\rho}(0)}{2(\nu+1)}\left(\frac{\sqrt{\a}-\sqrt{\b}}{\sqrt{\a}+\sqrt{\b}} \right)^{\nu+1}-\int_{0}^{\infty}Q'_{\rho}(t)I_{\nu+1}\left(\pi \sqrt{t}\left(\sqrt{\a}-\sqrt{\b}\right)\right)K_{\nu+1}\left(\pi \sqrt{t}\left(\sqrt{\a}+\sqrt{\b}\right)\right)dt,\label{maintheoremequation}
\end{align}
where it is assumed that $Q_{\rho}(0)$ exists, and where ${_2F_1}(a,b;c;z)$ denotes the ordinary hypergeometric function.
\end{theorem}

\begin{proof} Replace $x$ by $t$ in \eqref{21},
multiply both sides of \eqref{21} by
\begin{align}
\mathbb{I}_{\nu+1}(\a,\b, t):=\frac{d}{dt}\left\{I_{\nu+1}\left(\pi \sqrt{t}\left(\sqrt{\a}-\sqrt{\b}\right)\right)K_{\nu+1}\left(\pi \sqrt{t}\left(\sqrt{\a}+\sqrt{\b}\right)\right) \right\},\label{b11}
\end{align}
 and finally integrate with respect to $t$ over $(0,\infty)$. We see that the left-hand side becomes
\begin{align}\label{withrho}
&\frac{1}{\Gamma(\rho+1)}\sum_{n=1}^{\infty}a(n)\int_{\l_n}^{\infty}(t-\l_n)^\rho\mathbb{I}_{\nu+1}(\a,\b, t)dt=\frac{1}{(2\pi)^{\rho}}F_1(\a,\b,\rho)+F_{2}(\a,\b,\rho),
\end{align}
where
\begin{align}
F_{1}(\a,\b,\rho)
&:=\sum_{n=1}^{\infty}b(n)\int_{0}^{\infty}\left(\df{t}{\mu_n}\right)^{(\delta+\rho)/2}
J_{\delta+\rho}\left(4\pi\sqrt{\mu_n t}\right)~\mathbb{I}_{\nu+1}(\a,\b, t) dt\notag\\
\intertext{and}
F_2(\a,\b,\rho)&:=\int_{0}^{\infty}Q_{\rho}(t)\mathbb{I}_{\nu+1}(\a,\b, t)dt.\label{b8}
\end{align}

First, examine $F_1(\alpha,\beta,\rho)$.  Integrating by parts while using \eqref{b3a} in the form
\begin{align*}
\frac{d}{dt}\left({t}^{(\delta+\rho)/2}J_{\delta+\rho}\left(a\sqrt{ t}\right)\right)=\frac{a}{2}{t}^{(\delta+\rho-1)/2}J_{\delta+\rho-1}\left(a\sqrt{ t}\right),
\end{align*}
we find that
\begin{align}
&\qquad F_1(\alpha,\beta,\rho)\notag\\
=&\sum_{n=1}^{\infty}\frac{b(n)}{\mu_n^{(\delta+\rho)/2}}\int_{0}^{\infty}{t}^{(\delta+\rho)/2}
J_{\delta+\rho}\left(4\pi\sqrt{\mu_n t}\right)~\frac{d}{dt}\left\{I_{\nu+1}\left(\pi \sqrt{t}\left(\sqrt{\a}-\sqrt{\b}\right)\right)K_{\nu+1}\left(\pi \sqrt{t}\left(\sqrt{\a}+\sqrt{\b}\right)\right) \right\} dt\nonumber\\
=&-\sum_{n=1}^{\infty}\frac{b(n)}{\mu_n^{(\delta+\rho)/2}}\int_{0}^{\infty}\frac{d}{dt}
\left({t}^{(\delta+\rho)/2}J_{\delta+\rho}\left(4\pi\sqrt{\mu_n t}\right)\right)I_{\nu+1}\left(\pi \sqrt{t}\left(\sqrt{\a}-\sqrt{\b}\right)\right)K_{\nu+1}\left(\pi \sqrt{t}\left(\sqrt{\a}+\sqrt{\b}\right)\right) dt\nonumber\\
=&-2\pi\sum_{n=1}^{\infty}\frac{b(n)}{\mu_n^{(\delta+\rho-1)/2}}\int_{0}^{\infty}t^{(\delta+\rho-1)/2}
J_{\delta+\rho-1}\left(4\pi\sqrt{\mu_n t}\right)I_{\nu+1}\left(\pi \sqrt{t}\left(\sqrt{\a}-\sqrt{\b}\right)\right)K_{\nu+1}\left(\pi \sqrt{t}\left(\sqrt{\a}+\sqrt{\b}\right)\right) dt,\label{b88}
\end{align}
where we have used the asymptotic formulas \eqref{b5a}--\eqref{b6}, the hypothesis $\delta+\rho>0$, and the existence of
\begin{equation*}
\lim_{t\to 0}I_{\nu+1}\left(\pi \sqrt{t}\left(\sqrt{\a}-\sqrt{\b}\right)\right)K_{\nu+1}\left(\pi \sqrt{t}\left(\sqrt{\a}+\sqrt{\b}\right)\right),
\end{equation*}
which is explicitly calculated in \eqref{b12} below. 

We next employ an integral evaluation from  \cite[p.~315]{bdkz}.  For $\Re(\mu)>-1$, $\Re(\mu+\nu)>-1$, and
$\Re(\pi(z+w))>|\Re(\pi(z-w))|+|\Im(\xi)|$,
\begin{align}\label{bdkzintegral}
&\int_0^{\infty}x^{\mu+1}J_{\mu}(\xi x)I_{\nu}(\pi(z-w)x)K_{\nu}(\pi(z+w)x)dx\notag\\
=&\df{\Gamma(\mu+\nu+1)}{\Gamma(\nu+1)}\df{(\xi/2)^{\mu}}{\sqrt{\xi^2+4\pi^2z^2}\sqrt{\xi^2+4\pi^2w^2}}
\left(\df{\sqrt{\xi^2+4\pi^2z^2}-\sqrt{\xi^2+4\pi^2w^2}}
{\sqrt{\xi^2+4\pi^2z^2}+\sqrt{\xi^2+4\pi^2w^2}}\right)^{\nu}\notag\\
&\times\left(\df{1}{\sqrt{\xi^2+4\pi^2z^2}}+\df{1}{\sqrt{\xi^2+4\pi^2w^2}}\right)^{2\mu}\notag\\
&\times {_2F_1}\left(\nu-\mu,-\mu;\nu+1;
 \left(\df{\sqrt{\xi^2+4\pi^2z^2}-\sqrt{\xi^2+4\pi^2w^2}}
 {\sqrt{\xi^2+4\pi^2z^2}+\sqrt{\xi^2+4\pi^2w^2}}\right)^2\right).
\end{align}
The Hankel inversion of the formula given above with the same kernel, that is, $J_{\mu}$, was given by Koshliakov \cite[Equation (1)]{kosh34} and is a generalization of an integral evaluation by V.~A.~Fock and V.~Bursian \cite[pp.~361--363]{fockbursian} arising in their study on electromagnetism (see also  \cite[Equations~(31), (33)]{fock}).

In the integral on the extreme right-hand side of \eqref{b88}, make the change of variable $t=x^2$ and then apply
\eqref{bdkzintegral} with $\xi=4\pi\sqrt{\mu_n}$, $\mu=\delta+\rho-1$, $z=\sqrt{\alpha}$, $w=\sqrt{\beta}$, and $\nu$ replaced by $\nu+1$.  Thus, for $\Re(\nu+\delta+\rho)>-1$ and $\delta+\rho>0$, we find that
\begin{align}\label{integralF1}
&F_1(\a,\b,\rho)=-2\pi\sum_{n=1}^{\infty}\frac{b(n)}{\mu_{n}^{(\delta+\rho-1)/2}}
\Bigg\{\frac{2\Gamma(\nu+\delta+\rho+1)}{\Gamma(\nu+2)}
\frac{(2\pi\sqrt{\mu_n})^{\delta+\rho-1}}{\sqrt{16\pi^2\mu_n+4\pi^2\a}\sqrt{16\pi^2\mu_n+4\pi^2\b}}\nonumber\\
&\quad\times\left(\frac{\sqrt{4\mu_n+\a}-\sqrt{4\mu_n+\b}}{\sqrt{4\mu_n+\a}+\sqrt{4\mu_n+\b}} \right)^{\nu+1}\left(\frac{1}{\sqrt{16\pi^2\mu_n+4\pi^2\a}} +\frac{1}{\sqrt{16\pi^2\mu_n+4\pi^2\b}}\right)^{2\delta+2\rho-2}\nonumber\\
&\quad\times{}_2F_{1}\left[\nu-\delta-\rho+2,-\delta-\rho+1,\nu+2,
\left(\frac{\sqrt{4\mu_n+\a}-\sqrt{4\mu_n+\b}}{\sqrt{4\mu_n+\a}+\sqrt{4\mu_n+\b}} \right)^{2} \right]\Bigg\}\nonumber\\
&=-\frac{2}{(2\pi)^{\d+\rho}}\frac{\Gamma(\nu+\delta+\rho+1)}
{\Gamma(\nu+2)}\sum_{n=1}^{\infty}\frac{b(n)}{\sqrt{4\mu_n+\a}
\sqrt{4\mu_n+\b}}\left(\frac{\sqrt{4\mu_n+\a}-\sqrt{4\mu_n+\b}}
{\sqrt{4\mu_n+\a}+\sqrt{4\mu_n+\b}} \right)^{\nu+1}\nonumber\\
&\quad\times\left(\frac{1}{\sqrt{4\mu_n+\a}} +\frac{1}{\sqrt{4\mu_n+\b}}\right)^{2\delta+2\rho-2}
{}_2F_{1}\left[\nu-\delta-\rho+2,-\delta-\rho+1;\nu+2;
\left(\frac{\sqrt{4\mu_n+\a}-\sqrt{4\mu_n+\b}}{\sqrt{4\mu_n+\a}+\sqrt{4\mu_n+\b}} \right)^{2} \right].
\end{align}

Second, from \eqref{b8} and \eqref{b11},
\begin{align}
F_2(\a,\b,\rho)
&=\int_{0}^{\infty}Q_{\rho}(t)\frac{d}{dt}\left\{I_{\nu+1}\left(\pi \sqrt{t}\left(\sqrt{\a}-\sqrt{\b}\right)\right)K_{\nu+1}\left(\pi \sqrt{t}\left(\sqrt{\a}+\sqrt{\b}\right)\right) \right\} dt.\label{integralF}
\end{align}
Now, by the definitions of $I_{\nu}$ and $K_{\nu}$ in \eqref{b1} and \eqref{b2}, respectively, and by the use of the functional equation and reflection formula for $\Gamma(z)$, we find that
\begin{align}\label{b12}
&\lim_{t\to 0}I_{\nu+1}\left(\pi \sqrt{t}\left(\sqrt{\a}-\sqrt{\b}\right)\right)K_{\nu+1}\left(\pi \sqrt{t}\left(\sqrt{\a}+\sqrt{\b}\right)\right)\notag\\
=&\lim_{t\to 0}\df{\pi}{2\sin(\pi (\nu+1))}\df{\left(\tfrac12\pi\sqrt{t}(\sqrt{\alpha}-\sqrt{\beta})\right)^{\nu+1}}{\Gamma(\nu+2)}
\df{\left(\tfrac12\pi\sqrt{t}(\sqrt{\alpha}+\sqrt{\beta})\right)^{-\nu-1}}{\Gamma(-\nu)}\notag\\
=&\df{\pi}{2\sin(\pi(\nu+1))(\nu+1)\Gamma(\nu+1)\Gamma(-\nu)}
\left(\df{\sqrt{\alpha}-\sqrt{\beta}}{\sqrt{\alpha}+\sqrt{\beta}}\right)^{\nu+1}\notag\\
=&\df{1}{2(\nu+1)}\left(\df{\sqrt{\alpha}-\sqrt{\beta}}{\sqrt{\alpha}+\sqrt{\beta}}\right)^{\nu+1}.
\end{align}
Utilizing \eqref{b12}, \eqref{b5}, and \eqref{b6} in performing an integration by parts in \eqref{integralF}, we deduce that, for $\Re(\nu)>-1$,
\begin{align}\label{integralF2}
F_2(\a,\b,\rho)=&-\frac{1}{2(\nu+1)}\left(\frac{\sqrt{\a}-\sqrt{\b}}{\sqrt{\a}+\sqrt{\b}} \right)^{\nu+1}Q_{\rho}(0)\notag\\&-\int_{0}^{\infty}Q'_{\rho}(t)I_{\nu+1}\left(\pi \sqrt{t}\left(\sqrt{\a}-\sqrt{\b}\right)\right)K_{\nu+1}\left(\pi \sqrt{t}\left(\sqrt{\a}+\sqrt{\b}\right)\right)dt,
\end{align}
where, for Re$(\sqrt{\a})>$ Re$(\sqrt{\b})$, the boundary term at $\infty$ vanishes, since by \eqref{b5} and \eqref{b6}, respectively, as $t\to\infty$,
\begin{equation}\label{icomu}
I_{\nu}(\pi(\sqrt{t\a}-\sqrt{t\b})\sim\frac{e^{\pi(\sqrt{t\a}-\sqrt{t\b})}}{\pi\sqrt{2(\sqrt{t\a}-\sqrt{t\b})}}
\end{equation}
and
\begin{align}\label{kcom}
K_{\nu}(\pi(\sqrt{t\a}+\sqrt{t\b}))\sim\frac{e^{-\pi(\sqrt{t\a}+\sqrt{t\b})}}{\sqrt{2(\sqrt{t\a}+\sqrt{t\b})}}.
\end{align}

Finally, from \eqref{withrho}, \eqref{integralF1}, and \eqref{integralF2}, we deduce that
\begin{align*}
&\frac{1}{\Gamma(\rho+1)}\sum_{n=1}^{\infty}a(n)\int_{\l_n}^{\infty}(t-\l_n)^\rho\frac{d}{dt}\left\{I_{\nu+1}\left(\pi \sqrt{t}\left(\sqrt{\a}-\sqrt{\b}\right)\right)K_{\nu+1}\left(\pi \sqrt{t}\left(\sqrt{\a}+\sqrt{\b}\right)\right) \right\}dt\nonumber\\
=&-\frac{1}{2(\nu+1)}\left(\frac{\sqrt{\a}-\sqrt{\b}}{\sqrt{\a}+\sqrt{\b}} \right)^{\nu+1}Q_{\rho}(0)\notag\\&-\frac{2}{(2\pi)^{\d+2\rho}}
\frac{\Gamma(\nu+\delta+\rho+1)}{\Gamma(\nu+2)}
\sum_{n=1}^{\infty}\frac{b(n)}{\sqrt{4\mu_n+\a}
\sqrt{4\mu_n+\b}}\left(\frac{\sqrt{4\mu_n+\a}-\sqrt{4\mu_n+\b}}{\sqrt{4\mu_n+\a}+\sqrt{4\mu_n+\b}} \right)^{\nu+1}\nonumber\\
&\times\left(\frac{1}{\sqrt{4\mu_n+\a}} +\frac{1}{\sqrt{4\mu_n+\b}}\right)^{2\delta+2\rho-2}
{}_2F_{1}\left[\nu-\delta-\rho+2,1-\delta-\rho;\nu+2;
\left(\frac{\sqrt{4\mu_n+\a}-\sqrt{4\mu_n+\b}}{\sqrt{4\mu_n+\a}+\sqrt{4\mu_n+\b}} \right)^{2} \right]\nonumber\\
&-\int_{0}^{\infty}Q'_{\rho}(t)I_{\nu+1}\left(\pi \sqrt{t}\left(\sqrt{\a}-\sqrt{\b}\right)\right)K_{\nu+1}\left(\pi \sqrt{t}\left(\sqrt{\a}+\sqrt{\b}\right)\right)dt.\notag
\end{align*}
The proof of Theorem \ref{maintheoremgeneral} is now complete.
\end{proof}

\section{The Special Case: $\rho=0$}
When $\rho=0$ in Theorem \ref{maintheoremgeneral}, by \eqref{icomu} and \eqref{kcom}, the left-hand side of \eqref{maintheoremequation} reduces to
\begin{align*}
&\sum_{n=1}^{\infty}a(n)\int_{\l_n}^{\infty}\frac{d}{dt}\left\{I_{\nu+1}\left(\pi\sqrt{t}
\left(\sqrt{\a}-\sqrt{\b}\right)\right)K_{\nu+1}\left(\pi\sqrt{t}\left(\sqrt{\a}+\sqrt{\b}\right)\right) \right\}dt\nonumber\\
&=-\sum_{n=1}^{\infty}a(n)I_{\nu+1}\left(\pi \sqrt{\l_n}\left(\sqrt{\a}-\sqrt{\b}\right)\right)K_{\nu+1}\left(\pi \sqrt{\l_n}\left(\sqrt{\a}+\sqrt{\b}\right)\right).
\end{align*}
Hence, we have our second main theorem.

\begin{theorem}\label{maintheorem2}
Assume that $\Re(\nu)>-1$ and $\Re(\sqrt{\alpha})>\Re(\sqrt{\beta})>0$.  Also assume that  $\delta+\Re(\nu)+1>\sigma_a^*>0$.  Suppose that the integral on the right side below converges absolutely and that $Q_0(0)$ exists.
Then,
\begin{align}
&\sum_{n=1}^{\infty}a(n)I_{\nu+1}\left(\pi \sqrt{\l_n}\left(\sqrt{\a}-\sqrt{\b}\right)\right)K_{\nu+1}\left(\pi \sqrt{\l_n}\left(\sqrt{\a}+\sqrt{\b}\right)\right)\nonumber\\
=&\frac{2(2\pi)^{-\delta}\Gamma(\nu+\delta+1)}
{\Gamma(\nu+2)}\sum_{n=1}^{\infty}\frac{b(n)}
{\sqrt{4\mu_n+\a}\sqrt{4\mu_n+\b}}
\left(\frac{\sqrt{4\mu_n+\a}-\sqrt{4\mu_n+\b}}{\sqrt{4\mu_n+\a}+\sqrt{4\mu_n+\b}} \right)^{\nu+1}\nonumber\\
&\times\left(\frac{1}{\sqrt{4\mu_n+\a}} +\frac{1}{\sqrt{4\mu_n+\b}}\right)^{2\delta-2}
{}_2F_{1}\left[\nu-\delta+2,1-\delta;\nu+2;
\left(\frac{\sqrt{4\mu_n+\a}-\sqrt{4\mu_n+\b}}{\sqrt{4\mu_n+\a}+\sqrt{4\mu_n+\b}} \right)^{2} \right]\nonumber\\
&+\frac{Q_{0}(0)}{2(\nu+1)}\left( \frac{\sqrt{\a}-\sqrt{\b}}{\sqrt{\a}+\sqrt{\b}}\right)^{\nu+1}+\int_{0}^{\infty}Q'_{0}(x)I_{\nu+1}\left(\pi \sqrt{x}\left(\sqrt{\a}-\sqrt{\b}\right)\right)K_{\nu+1}\left(\pi \sqrt{x}\left(\sqrt{\a}+\sqrt{\b}\right)\right)dx.\label{b13}
\end{align}
\end{theorem}

Next, we show that Theorem 4.1 from \cite{BDGZ1} can be obtained as a special case of Theorem \ref{maintheorem2}. To that end, divide both sides of \eqref{b13} by $(\sqrt{\a}-\sqrt{\b})^{\nu+1}$, and let $\a\to \b$. In the course of doing so, we need the limit
\begin{align}\label{b14}
\lim_{\a\to \b}\frac{I_{\nu+1}\left(\pi \sqrt{\l_n}\left(\sqrt{\a}-\sqrt{\b}\right)\right)K_{\nu+1}\left(\pi \sqrt{\l_n}\left(\sqrt{\a}+\sqrt{\b}\right)\right)}{(\sqrt{\a}-\sqrt{\b})^{\nu+1}}=\left(\frac{\pi}{2} \right)^{\nu+1}\lambda_n^{(\nu+1)/2}\frac{K_{\nu+1}(2\pi\sqrt{\l_n\b})}{\Gamma(\nu+2)},
\end{align}
where the definitions of $I_{\nu}$ and $K_{\nu}$ in \eqref{b1} and \eqref{b2}, respectively, were used.
On the left side of \eqref{b13}, by \eqref{b3} and \eqref{b4}, the series converges absolutely and uniformly with respect to $\alpha$ for $0\leq \sqrt{\alpha}<\epsilon$, for each fixed $\epsilon>0$. Thus, we can interchange summation and the limit as $\alpha\to\beta$ on the left-hand side of \eqref{b13} to find that
\begin{align}\label{b17}
&\lim_{\alpha\to\beta}\sum_{n=1}^{\infty}a(n)I_{\nu+1}\left(\pi \sqrt{\l_n}\left(\sqrt{\a}-\sqrt{\b}\right)\right)K_{\nu+1}\left(\pi \sqrt{\l_n}\left(\sqrt{\a}+\sqrt{\b}\right)\right)\notag\\=&
\frac{\left(\frac{\pi}{2} \right)^{\nu+1}}{\Gamma(\nu+2)}\sum_{n=1}^{\infty}a(n)\lambda_{n}^{(\nu+1)/2}K_{\nu+1}(2\pi\sqrt{\l_n\b}).
\end{align}
We also take the limit as $\alpha\to\beta$ inside the integral on the far right side of \eqref{b13} by using a similar argument with $\lambda_n$ replaced by $x$ in \eqref{b14}.  Hence,
\begin{align}\label{b15}
&\lim_{\alpha\to\beta}\int_{0}^{\infty}Q'_{0}(x)I_{\nu+1}\left(\pi \sqrt{x}\left(\sqrt{\a}-\sqrt{\b}\right)\right)K_{\nu+1}\left(\pi \sqrt{x}\left(\sqrt{\a}+\sqrt{\b}\right)\right)dx\notag\\
=&\df{\left(\frac{\pi}{2}\right)^{\nu+1}}{\Gamma(\nu+2)}\int_0^{\infty}Q_0^{\prime}(x)x^{(\nu+1)/2}
K_{\nu+1}(2\pi\sqrt{\beta x})dx.
\end{align}
Next, recall that ${_2F_1}(a,b;c;0)=1$.  Thus, it remains to evaluate the limit
\begin{align}\label{b16}
&\lim_{\alpha\to\beta}\frac{\left(\sqrt{4\mu_n+\a}+\sqrt{4\mu_n+\b}\right)^{-\nu-1}}
{\sqrt{4\mu_n+\a}\sqrt{4\mu_n+\b}}
\left(\frac{\sqrt{4\mu_n+\a}-\sqrt{4\mu_n+\b}}
{\sqrt{\alpha}-\sqrt{\beta}} \right)^{\nu+1}
\left(\frac{1}{\sqrt{4\mu_n+\a}} +\frac{1}{\sqrt{4\mu_n+\b}}\right)^{2\delta-2}\notag\\
&=\lim_{\alpha\to\beta}\frac{\left(\sqrt{4\mu_n+\a}+\sqrt{4\mu_n+\b}\right)^{-2\nu-2}}
{\sqrt{4\mu_n+\a}\sqrt{4\mu_n+\b}}(\sqrt{\a}+\sqrt{\b})^{\nu+1}\left(\frac{1}{\sqrt{4\mu_n+\a}} +\frac{1}{\sqrt{4\mu_n+\b}}\right)^{2\delta-2}\notag\\
&=\df{2^{2\delta-2}}{(4\mu_n+\beta)^{\delta}}\left(\df{\sqrt{\beta}}
{2(4\mu_n+\beta)}\right)^{\nu+1}.
\end{align}

Bringing together \eqref{b17}--\eqref{b16}, we conclude that
\begin{align}\label{b18}
&\frac{\left(\frac{\pi}{2} \right)^{\nu+1}}{\Gamma(\nu+2)}\sum_{n=1}^{\infty}a(n)\lambda_{n}^{(\nu+1)/2}K_{\nu+1}(2\pi\sqrt{\l_n\b})
=\frac{2^{\delta-\nu-2}\pi^{-\delta}\beta^{(\nu+1)/2}\Gamma(\nu+\delta+1)}
{\Gamma(\nu+2)}\sum_{n=1}^{\infty}\frac{b(n)}{(4\mu_n+\b)^{\delta+\nu+1}}\nonumber\\
&\qquad+\frac{\left(\frac{\pi}{2} \right)^{\nu+1}}{\Gamma(\nu+2)}\int_{0}^{\infty}Q_{0}^{\prime}(x)x^{(\nu+1)/2}K_{\nu+1}(2\pi\sqrt{\beta x})dx.
\end{align}
  Let $s=2\pi\sqrt{\beta}$.  Multiplying both sides of \eqref{b18} by $2/s$ and by $\left(\frac{\pi}{2} \right)^{-\nu-1}\Gamma(\nu+2)$ and then integrating by parts with the aid of \eqref{b4} and \eqref{b6}, we conclude that
\begin{align*}
\frac{2}{s}\sum_{n=1}^{\infty}a(n)\lambda_{n}^{(\nu+1)/2}K_{\nu+1}(s\sqrt{\l_n})
=&2^{3\delta+\nu+1}\pi^{\delta}s^\nu\Gamma(\nu+\delta+1)
\sum_{n=1}^{\infty}\frac{b(n)}{(16\pi^2\mu_n+s^2)^{\delta+\nu+1}}\nonumber\\
&+\df{2}{s}\int_{0}^{\infty}Q_{0}^{\prime}(x)x^{(\nu+1)/2}K_{\nu+1}(s\sqrt{x})dx\notag\\
=&2^{3\delta+\nu+1}\pi^{\delta}s^\nu\Gamma(\nu+\delta+1)
\sum_{n=1}^{\infty}\frac{b(n)}{(16\pi^2\mu_n+s^2)^{\delta+\nu+1}}\nonumber\\
&+\int_{0}^{\infty}Q_{0}(x)x^{\nu/2}K_{\nu}(s\sqrt{x})dx.
\end{align*}

We state this last result as a corollary.

\begin{corollary}\label{cor} For $\Re(\nu)>-1$, $\delta+\Re(\nu)+1>\sigma_a^*$, and $\Re(s)>0$,
\begin{align*}
\frac{2}{s}\sum_{n=1}^{\infty}a(n)\lambda_{n}^{(\nu+1)/2}K_{\nu+1}(s\sqrt{\l_n})
=&2^{3\delta+\nu+1}\pi^{\delta}s^\nu\Gamma(\nu+\delta+1)
\sum_{n=1}^{\infty}\frac{b(n)}{(16\pi^2\mu_n+s^2)^{\delta+\nu+1}}\nonumber\\
&+\int_{0}^{\infty}Q_{0}(x)x^{\nu/2}K_{\nu}(s\sqrt{x})dx,
\end{align*}
where it is assumed that the integral converges absolutely.
\end{corollary}

Corollary \ref{cor}  was also established in \cite[Theorem 4.1]{BDGZ1}.

\section{Example: $r_k(n)$}
Recall \eqref{functionalequation}--\eqref{Q}. Applying Theorem \ref{maintheorem2} with $\a$ and $\b$ replaced by $2\a$ and $2\b$, respectively, and $\nu$ replaced by $\nu-1$, for $\Re(\nu)>0$, we find that
\begin{align}\label{rkn}
&\sum_{n=1}^{\infty}r_k(n)I_{\nu}\left(\pi \sqrt{n}\left(\sqrt{\a}-\sqrt{\b}\right)\right)K_{\nu}\left(\pi \sqrt{n}\left(\sqrt{\a}+\sqrt{\b}\right)\right)=\nonumber\\
=&\frac{\Gamma(k/2+\nu)}{\pi^{k/2}2^{k-1}\Gamma(\nu+1)}
\sum_{n=1}^{\infty}\frac{b(n)}{\sqrt{n+\a}\sqrt{n+\b}}
\left(\frac{\sqrt{n+\a}-\sqrt{n+\b}}{\sqrt{n+\a}+\sqrt{n+\b}} \right)^{\nu}\nonumber\\
&\times\left(\frac{1}{\sqrt{n+\a}} +\frac{1}{\sqrt{n+\b}}\right)^{k-2}{}_2F_{1}
\left[\nu-k/2+1,1-k/2;\nu+1;\left(\frac{\sqrt{n+\a}-\sqrt{n+\b}}{\sqrt{n+\a}+\sqrt{n+\b}} \right)^{2} \right]\nonumber\\
&-\frac{1}{2\nu}\left( \frac{\sqrt{\a}-\sqrt{\b}}{\sqrt{\a}+\sqrt{\b}}\right)^{\nu}+\int_{0}^{\infty}Q'_{0}(x)I_{\nu}\left(\pi \sqrt{2x}\left(\sqrt{\a}-\sqrt{\b}\right)\right)K_{\nu}\left(\pi \sqrt{2x}\left(\sqrt{\a}+\sqrt{\b}\right)\right)dx.
\end{align}
To evaluate this integral, we use an integral in \cite[p.~717, Equation 6.576, no.~5]{gr}, namely, for $a>b,$ $\Re(2\nu)>\l-1$, and $\Re(\l)<1,$
\begin{align}\label{gr}
\int_{0}^{\infty}x^{-\l}K_{\nu}(ax)I_{\nu}(bx)dx=\frac{b^{\nu}\Gamma\left(\frac{1-\l+2\nu}{2} \right)\Gamma\left(\frac{1-\l}{2} \right)}{2^{\l+1}\Gamma(\nu+1)a^{1-\l+\nu}}
{}_2F_1\left(\frac{1-\l+2\nu}{2},\frac{1-\l}{2};\nu+1;\frac{b^2}{a^2}\right).
\end{align}
Using  \eqref{Q} and \eqref{gr}, wherein we  make the change of variable $t=\sqrt{2x}$ and note that $\lambda=1-k$, $a=\pi(\sqrt{\alpha}+\sqrt{b})$, and $b=\pi(\sqrt{\alpha}-\sqrt{b})$, we deduce that
\begin{align}
&\int_{0}^{\infty}Q'_{0}(x)I_{\nu}\left(\pi \sqrt{2x}\left(\sqrt{\a}-\sqrt{\b}\right)\right)K_{\nu}\left(\pi \sqrt{2x}\left(\sqrt{\a}+\sqrt{\b}\right)\right)dx\nonumber\\
&=\frac{k(2\pi)^{k/2}}{2\Gamma(1+k/2)}\int_{0}^{\infty}x^{k/2-1}I_{\nu}\left(\pi \sqrt{2x}\left(\sqrt{\a}-\sqrt{\b}\right)\right)K_{\nu}\left(\pi \sqrt{2x}\left(\sqrt{\a}+\sqrt{\b}\right)\right)dx\nonumber\\
&=\frac{k2^{k/2-1}\pi^{k/2}}{2^{k/2-1}\Gamma(1+k/2)}\int_{0}^{\infty}t^{k-1}I_{\nu}\left(\pi t\left(\sqrt{\a}-\sqrt{\b}\right)\right)K_{\nu}\left(\pi t\left(\sqrt{\a}+\sqrt{\b}\right)\right)dt\nonumber\\
&=\frac{k\pi^{k/2}}{\Gamma(1+k/2)}
\frac{\left(\pi\left(\sqrt{\a}-\sqrt{\b}\right)\right)^{\nu}\Gamma\left(\frac{k}{2}+\nu \right)\Gamma\left(\frac{k}{2} \right)}{2^{2-k}\Gamma(\nu+1)\left(\pi \left(\sqrt{\a}+\sqrt{\b}\right)\right)^{\nu+k}}\cdot
{}_2F_1\left(\frac{k}{2}+\nu,\frac{k}{2};\nu+1;\left(\frac{\sqrt{\a}-\sqrt{\b}}{\sqrt{\a}+\sqrt{\b}} \right)^{2} \right)\nonumber\\
&=\frac{2^{k-1}\Gamma\left(\frac{k}{2}+\nu \right)}{\pi^{k/2}\Gamma(\nu+1)}\left(\frac{\sqrt{\a}-\sqrt{\b}}{\sqrt{\a}+\sqrt{\b}} \right)^{\nu}\left(\frac{1}{\sqrt{\a}+\sqrt{\b}} \right)^{k}{}_2F_1\left(\frac{k}{2}+\nu,\frac{k}{2};\nu+1;\left(\frac{\sqrt{\a}-\sqrt{\b}}{\sqrt{\a}+\sqrt{\b}} \right)^{2} \right).\label{ex1b}
\end{align}
Invoking Euler's formula \cite[p.~68, Theorem 2.2.5]{aar},
\begin{align*}
{}_2F_1\left(a,b;c;x \right)=(1-x)^{c-a-b}{}_2F_1\left(c-a,c-b;c;x \right).
\end{align*}
in \eqref{ex1b}, we find that
\begin{align}
&\int_{0}^{\infty}Q'_{0}(x)I_{\nu}\left(\pi \sqrt{2x}\left(\sqrt{\a}-\sqrt{\b}\right)\right)K_{\nu}\left(\pi \sqrt{2x}\left(\sqrt{\a}+\sqrt{\b}\right)\right)dx\nonumber\\
&=\frac{2^{k-1}\Gamma\left(\frac{k}{2}+\nu \right)}{\pi^{k/2}\Gamma(\nu+1)}\left(\frac{\sqrt{\a}-\sqrt{\b}}{\sqrt{\a}+\sqrt{\b}} \right)^{\nu}\left(\frac{1}{\sqrt{\a}+\sqrt{\b}} \right)^{k}{}_2F_1\left(\frac{k}{2}+\nu,\frac{k}{2};\nu+1;\left(\frac{\sqrt{\a}-\sqrt{\b}}{\sqrt{\a}+\sqrt{\b}} \right)^{2} \right)\nonumber\\
&=\frac{2^{1-k}\Gamma\left(\frac{k}{2}+\nu \right)(\sqrt{\a\b})^{1-k}}{\pi^{k/2}\Gamma(\nu+1)}\left(\frac{\sqrt{\a}-\sqrt{\b}}{\sqrt{\a}+\sqrt{\b}} \right)^{\nu}\left(\frac{1}{\sqrt{\a}+\sqrt{\b}} \right)^{2-k}{}_2F_1\left(1-\frac{k}{2}+\nu,1-\frac{k}{2};\nu+1;
\left(\frac{\sqrt{\a}-\sqrt{\b}}{\sqrt{\a}+\sqrt{\b}} \right)^{2}\right)\nonumber\\
&=\frac{\Gamma\left(\frac{k}{2}+\nu \right)}{\pi^{k/2}2^{k-1}\Gamma(\nu+1)}\left(\frac{\sqrt{\a}-\sqrt{\b}}{\sqrt{\a}+\sqrt{\b}} \right)^{\nu}\left(\frac{1}{\sqrt{\a}}+\frac{1}{\sqrt{\b}} \right)^{k-2}\frac{1}{\sqrt{\a\b}}\cdot
{}_2F_1\left(1-\frac{k}{2}+\nu,1-\frac{k}{2};\nu+1;\left(\frac{\sqrt{\a}-\sqrt{\b}}{\sqrt{\a}+\sqrt{\b}} \right)^{2} \right).\label{ex1c}
\end{align}
Now put \eqref{ex1c} in \eqref{rkn}.  To obtain the final equality below, we define
 $r_k(0)=1$.  To that end,  
\begin{align}\label{rkn1}
&\sum_{n=1}^{\infty}r_k(n)I_{\nu}\left(\pi \sqrt{n}\left(\sqrt{\a}-\sqrt{\b}\right)\right)K_{\nu}\left(\pi \sqrt{n}\left(\sqrt{\a}+\sqrt{\b}\right)\right)\nonumber\\
&\quad=\frac{\Gamma(k/2+\nu)}{\pi^{k/2}2^{k-1}
\Gamma(\nu+1)}\sum_{n=1}^{\infty}\frac{r_k(n)}
{\sqrt{n+\a}\sqrt{n+\b}}\left(\frac{\sqrt{n+\a}-\sqrt{n+\b}}{\sqrt{n+\a}+\sqrt{n+\b}} \right)^{\nu}\nonumber\\
&\qquad\times\left(\frac{1}{\sqrt{n+\a}} +\frac{1}{\sqrt{n+\b}}\right)^{k-2}{}
_2F_{1}\left[\nu-k/2+1,1-k/2;\nu+1;\left(\frac{\sqrt{n+\a}-\sqrt{n+\b}}{\sqrt{n+\a}+\sqrt{n+\b}} \right)^{2} \right]\nonumber\\
&\qquad-\frac{1}{2\nu}\left( \frac{\sqrt{\a}-\sqrt{\b}}{\sqrt{\a}+\sqrt{\b}}\right)^{\nu}+\frac{\Gamma\left(\frac{k}{2}+\nu \right)}{\pi^{k/2}2^{k-1}\Gamma(\nu+1)}\left(\frac{\sqrt{\a}-\sqrt{\b}}{\sqrt{\a}+\sqrt{\b}} \right)^{\nu}\left(\frac{1}{\sqrt{\a}}+\frac{1}{\sqrt{\b}} \right)^{k-2}\frac{1}{\sqrt{\a\b}}\nonumber\\
&\qquad\times{}_2F_1\left(1-\frac{k}{2}+\nu,1-\frac{k}{2};\nu+1;
\left(\frac{\sqrt{\a}-\sqrt{\b}}{\sqrt{\a}+\sqrt{\b}} \right)^{2} \right)\nonumber\\
&\quad=-\frac{1}{2\nu}\left( \frac{\sqrt{\a}-\sqrt{\b}}{\sqrt{\a}+\sqrt{\b}}\right)^{\nu}
+\frac{\Gamma(k/2+\nu)}{\pi^{k/2}2^{k-1}\Gamma(\nu+1)}
\sum_{n=0}^{\infty}\frac{r_k(n)}{\sqrt{n+\a}\sqrt{n+\b}}
\left(\frac{\sqrt{n+\a}-\sqrt{n+\b}}{\sqrt{n+\a}+\sqrt{n+\b}} \right)^{\nu}\nonumber\\
&\qquad\times\left(\frac{1}{\sqrt{n+\a}} +\frac{1}{\sqrt{n+\b}}\right)^{k-2}{}_2F_1\left(1-\frac{k}{2}+\nu,1-\frac{k}{2};\nu+1;
\left(\frac{\sqrt{n+\a}-\sqrt{n+\b}}{\sqrt{n+\a}+\sqrt{n+\b}} \right)^{2} \right).
\end{align}

By a different method, the identity \eqref{rkn1} was also established in \cite[Theorem 1.6]{bdkz}.

\section{Example: Ramanujan's tau-function $\tau(n)$}

Let $\t(n)$ denote Ramanujan's famous arithmetical tau-function. Recall the associated facts and parameters given in
\eqref{50a}--\eqref{exam4}.
Then, from Theorem \ref{maintheorem2}, for $\Re(\nu)>-13/2$,
\begin{align}
&\sum_{n=1}^{\infty}\t(n)I_{\nu+1}\left(\pi \sqrt{n}\left(\sqrt{\a}-\sqrt{\b}\right)\right)K_{\nu+1}\left(\pi \sqrt{n}\left(\sqrt{\a}+\sqrt{\b}\right)\right)\nonumber\\
&\quad=\frac{2(2\pi)^{-12}\Gamma(13+\nu)}{\Gamma(\nu+2)
}\sum_{n=1}^{\infty}\frac{\tau(n)}{\sqrt{4n+\a}\sqrt{4n+\b}}
\left(\frac{\sqrt{4n+\a}-\sqrt{4n+\b}}{\sqrt{4n+\a}+\sqrt{4n+\b}} \right)^{\nu+1}\nonumber\\
&\qquad\times\left(\frac{1}{\sqrt{4n+\a}} +\frac{1}{\sqrt{4n+\b}}\right)^{22}
{}_2F_{1}\left[\nu-10,-11;\nu+2;\left(\frac{\sqrt{4n+\a}-\sqrt{4n+\b}}{\sqrt{4n+\a}+\sqrt{4n+\b}} \right)^{2} \right].\label{c1}
\end{align}
Letting $\nu=-\frac12$ in \eqref{c1} and using \eqref{I1} and \eqref{I2},
we are led to
\begin{align}\label{tau123}
&\frac{1}{\pi\sqrt{\a-\b}}\sum_{n=1}^{\infty}\frac{\tau(n)}{\sqrt{n}}
e^{-\pi\sqrt{n}(\sqrt{\a}+\sqrt{\b})}\sinh(\pi\sqrt{n}(\sqrt{\a}-\sqrt{\b}))\nonumber\\
=&\frac{2(2\pi)^{-12}\Gamma(25/2)}{\Gamma(3/2)}
\sum_{n=1}^{\infty}\frac{\tau(n)}{\sqrt{4n+\a}\sqrt{4n+\b}}
\left(\frac{\sqrt{4n+\a}-\sqrt{4n+\b}}{\sqrt{4n+\a}+\sqrt{4n+\b}} \right)^{1/2}\nonumber\\
&\times\left(\frac{1}{\sqrt{4n+\a}} +\frac{1}{\sqrt{4n+\b}}\right)^{22}
{}_2F_{1}\left[-21/2,-11;3/2;\left(\frac{\sqrt{4n+\a}-\sqrt{4n+\b}}{\sqrt{4n+\a}+\sqrt{4n+\b}} \right)^{2} \right].
\end{align}
Employing \cite[p.~461, no.~107]{pru}
\begin{align*}
{}_2F_{1}\left(a,a+1/2;3/2;z \right)=\df{1}{2(2a-1)\sqrt{z}}\left\{(1-\sqrt{z})^{1-2a}-(1+\sqrt{z})^{1-2a} \right\},
\end{align*}
with $a=-11$, we find that
\begin{equation}\label{c2}
{}_2F_{1}\left(-21/2,-11;3/2;z \right)=-\df{1}{46\sqrt{z}}\left\{(1-\sqrt{z})^{23}-(1+\sqrt{z})^{23} \right\}.
\end{equation}
With \eqref{c2} in \eqref{tau123} and with considerable simplification, we deduce that
\begin{gather}
\sum_{n=1}^{\infty}\frac{\tau(n)}{\sqrt{n}}e^{-\pi\sqrt{n}(\sqrt{\a}+\sqrt{\b})}\sinh(\pi\sqrt{n}(\sqrt{\a}-\sqrt{\b}))\notag\\
=2\frac{3\cdot5\cdots21}{\pi^{11}}\sum_{n=1}^{\infty}\tau(n)
\left(\frac{1}{(4n+\b)^{23/2}}-\frac{1}{(4n+\a)^{23/2}}\right).\label{tau-1/2}
\end{gather}
If we differentiate both sides of \eqref{tau-1/2} with respect to $\alpha$, let $\beta=\alpha$,  and simplify, we find that
\begin{equation*}
\sum_{n=1}^{\infty}\tau(n)e^{-2\pi\sqrt{n\alpha}}=2\frac{3\cdot5\cdots21\cdot23}{\pi^{12}}
\sum_{n=1}^{\infty}\frac{\sqrt{\alpha}\,\tau(n)}{(4n+\alpha)^{25/2}},
\end{equation*}
which, with $\alpha=s^2/(4\pi^2)$, gives \cite[Equation (7.4)]{BDGZ1}
\begin{equation*}
\sum_{n=1}^{\infty}\tau(n)e^{-s\sqrt{n}}
=2^{36}\pi^{23/2}\Gamma\left(\frac{25}{2}\right)\sum_{n=1}^{\infty}\df{s\tau(n)}{(s^2+16\pi^2n)^{25/2}}.
\end{equation*}

\section{Example: Primitive Dirichlet characters}

Let $\chi$ denote a primitive character modulo $q$. Depending on the parity of $\chi$, we separate two cases. First,  consider odd $\chi$.  Recall that the functional equations for the associated Dirichlet $L$-series is given in \eqref{funcequa1}, the Gauss sum $\tau(\chi)$ is defined in \eqref{gauss}, and the relevant parameters are given in \eqref{exam5}.
Consequently, by Theorem \ref{maintheorem2} and the fact that $Q_0(x)\equiv0$, for $\Re(\nu)>-5/2,$
\begin{align}\label{chifinal}
&\sum_{n=1}^{\infty}n\chi(n)I_{\nu+1}\left(\frac{\pi n}{\sqrt{2q}}\left(\sqrt{\a}-\sqrt{\b}\right)\right)K_{\nu+1}\left(\frac{\pi n}{\sqrt{2q}}\left(\sqrt{\a}+\sqrt{\b}\right)\right)\nonumber\\
&\quad=\frac{-i\pi^{-3/2}\Gamma(\nu+5/2)}{\sqrt{2q}\Gamma(\nu+2)}\tau(\chi)
\sum_{n=1}^{\infty}\frac{n\bar{\chi}(n)}{\sqrt{\left(\frac{2n^2}{q}+\a\right)}
\sqrt{\left(\frac{2n^2}{q}+\b\right)}}\left(\frac{\sqrt{\frac{2n^2}{q}+\a}
-\sqrt{\frac{2n^2}{q}+\b}}{\sqrt{\frac{2n^2}{q}+\a}+\sqrt{\frac{2n^2}{q}+\b}} \right)^{\nu+1}\nonumber\\
&\qquad\times\left(\frac{1}{\sqrt{\left(\frac{2n^2}{q}+\a\right)}} +\frac{1}{\sqrt{\left(\frac{2n^2}{q}+\b\right)}}\right)
{}_2F_{1}\left[\nu+1/2,-1/2;\nu+2;\left(\frac{\sqrt{\frac{2n^2}{q}+\a}-\sqrt{\frac{2n^2}{q}+\b}}
{\sqrt{\frac{2n^2}{q}+\a}+\sqrt{\frac{2n^2}{q}+\b}} \right)^{2} \right]\nonumber\\
&\quad=\frac{-i\pi^{-3/2}\Gamma(\nu+5/2)}{\sqrt{2q}\Gamma(\nu+2)}\tau(\chi)
\sum_{n=1}^{\infty}\frac{n\bar{\chi}(n)}{\left(\frac{2n^2}{q}
+\a\right)\left(\frac{2n^2}{q}+\b\right)}
\frac{\left(\sqrt{\frac{2n^2}{q}+\a}-\sqrt{\frac{2n^2}{q}+\b}\right)^{\nu+1}}
{\left(\sqrt{\frac{2n^2}{q}+\a}+\sqrt{\frac{2n^2}{q}+\b}\right)^{\nu}} \nonumber\\
&\qquad\times{}_2F_{1}\left[\nu+1/2,-1/2;\nu+2;\left(\frac{\sqrt{\frac{2n^2}{q}+\a}-
\sqrt{\frac{2n^2}{q}+\b}}{\sqrt{\frac{2n^2}{q}+\a}+\sqrt{\frac{2n^2}{q}+\b}} \right)^{2} \right].
\end{align}

Letting $\nu=-1/2$ in \eqref{chifinal}, using \eqref{I1} and \eqref{I2}, appealing to the trivial fact,
\begin{align*}
{}_2F_{1}\left(0,-1/2;3/2;x \right)=1,
\end{align*}
and multiplying both sides by $\pi\sqrt{\alpha-\beta}/\sqrt{2q}$, we deduce that
\begin{align}\label{chi2}
&\sum_{n=1}^{\infty}\chi(n)e^{-\frac{\pi n}{\sqrt{2q}}\left(\sqrt{\a}+\sqrt{\b}\right)}\sinh\left(\frac{\pi n}{\sqrt{2q}}\left(\sqrt{\a}-\sqrt{\b}\right)\right)\nonumber\\
&\quad=\frac{-i\pi^{-1/2}\Gamma(2)\sqrt{\a-\b}}{2q\Gamma(3/2)}
\tau(\chi)\sum_{n=1}^{\infty}\frac{n\bar{\chi}(n)}
{\left(\frac{2n^2}{q}+\a\right)\left(\frac{2n^2}{q}+\b\right)}
\frac{\left(\sqrt{\frac{2n^2}{q}+\a}-\sqrt{\frac{2n^2}{q}+\b}\right)^{1/2}}
{\left(\sqrt{\frac{2n^2}{q}+\a}+\sqrt{\frac{2n^2}{q}+\b}\right)^{-1/2}}\nonumber\\
&\quad=\frac{-iq\tau(\chi)\left(\a-\b\right)}{\pi}\sum_{n=1}^{\infty}\frac{n\bar{\chi}(n)}{\left(2n^2+\a q\right)\left(2n^2+\b q\right)}.
\end{align}

Next, let $\chi$ be even.  Recall that the functional equation and relevant parameters are given in \eqref{funcequa2} and \eqref{exam6}, respectively. Therefore, by Theorem \ref{maintheorem2}, for $\Re(\nu)>-3/2,$
\begin{align}\label{c6}
&\sum_{n=1}^{\infty}\chi(n)I_{\nu+1}\left(\frac{\pi n}{\sqrt{2q}}\left(\sqrt{\a}-\sqrt{\b}\right)\right)K_{\nu+1}\left(\frac{\pi n}{\sqrt{2q}}\left(\sqrt{\a}+\sqrt{\b}\right)\right)=\nonumber\\
&\quad=\frac{\sqrt{2}\Gamma(\nu+3/2)}{\sqrt{\pi q}\Gamma(\nu+2)}\tau(\chi)\sum_{n=1}^{\infty}\bar{\chi}(n)
\frac{\left(\sqrt{\frac{2n^2}{q}+\a}-\sqrt{\frac{2n^2}{q}+\b}\right)^{\nu+1}}
{\left(\sqrt{\frac{2n^2}{q}+\a}+\sqrt{\frac{2n^2}{q}+\b}\right)^{\nu+2}} \nonumber\\
&\qquad\times{}_2F_{1}\left[\nu+3/2,1/2;\nu+2;\left(\frac{\sqrt{\frac{2n^2}{q}+\a}
-\sqrt{\frac{2n^2}{q}+\b}}{\sqrt{\frac{2n^2}{q}+\a}+\sqrt{\frac{2n^2}{q}+\b}} \right)^{2} \right].
\end{align}
Letting $\nu=-1/2$ in \eqref{c6} and using the evaluation \cite[p.~1067, Formula 9.121, no.~7]{gr}
\begin{align}\label{log}
{}_2F_{1}\left(1,1/2;3/2;x \right)=\frac{1}{2\sqrt{x}}\log\left(\frac{1+\sqrt{x}}{1-\sqrt{x}} \right),
\end{align}
we obtain, after considerable simplification,
\begin{align}\label{14.5}
&\sum_{n=1}^{\infty}\frac{\chi(n)}{n}e^{-\frac{\pi n}{\sqrt{2q}}\left(\sqrt{\a}+\sqrt{\b}\right)}\sinh\left(\frac{\pi n}{\sqrt{2q}}\left(\sqrt{\a}-\sqrt{\b}\right)\right)
=\frac{\tau(\chi)}{2q}\sum_{n=1}^{\infty}\bar{\chi}(n)\log\left(\frac{2n^2+\a q}{2n^2+\b q}\right).
\end{align}

Equations \eqref{chifinal}, \eqref{chi2}, \eqref{c6}, and \eqref{14.5} are new.

\section{A Generalization of a Theorem of G.~N.~Watson}
The functional equation of the Riemann zeta function is given by \cite[p.~14]{edwards}
\begin{equation*}
\pi^{-s/2}\Gamma(s/2)\zeta(s)=\pi^{-(1-s)/2}\Gamma((1-s)/2)\zeta(1-s).
\end{equation*}
Hence, replacing $s$ by $2s$, we see that it can be transformed into the form \eqref{19} with $\delta=1/2$, $a(n)=b(n)=1$, and $\lambda_n=\mu_n=n^2/2$.
Note that in \eqref{22},
\begin{align*}
Q_0(x)=-\frac12+\sqrt{2x}.
\end{align*}
Employing \eqref{gr} with $x$ replaced by $\sqrt{x}$, and then with $\nu$ replaced by $\nu+1$, and letting $\lambda =0$, $a=\pi\left(\sqrt{\a}+\sqrt{\b}\right)$, and $b=\pi \left(\sqrt{\a}-\sqrt{\b}\right)$, we find that Theorem \ref{maintheorem2} yields, for $\Re(\nu)>-1/2,$
\begin{align*}
&\frac{1}{4(\nu+1)}\left( \frac{\sqrt{\a}-\sqrt{\b}}{\sqrt{\a}+\sqrt{\b}}\right)^{\nu+1}+\sum_{n=1}^{\infty}I_{\nu+1}\left(\frac{\pi n}{\sqrt2}\left(\sqrt{\a}-\sqrt{\b}\right)\right)K_{\nu+1}\left(\frac{\pi n}{\sqrt2}\left(\sqrt{\a}+\sqrt{\b}\right)\right)\nonumber\\
=&\frac{\Gamma(\nu+3/2)}{\sqrt{2\pi}\Gamma(\nu+2)} \frac{\left(\sqrt{\a}-\sqrt{\b}\right)^{\nu+1}}{\left(\sqrt{\a}+\sqrt{\b}\right)^{\nu+2}}\cdot
{}_2F_{1}\left(\nu+3/2,1/2;\nu+2;\left( \frac{\sqrt{\a}-\sqrt{\b}}{\sqrt{\a}+\sqrt{\b}}\right)^{2} \right)\nonumber\\
&+\frac{\sqrt{2}\Gamma(\nu+3/2)}{\sqrt{\pi}\Gamma(\nu+2)}
\sum_{n=1}^{\infty}\frac{\left(\sqrt{2n^2+\a}
-\sqrt{2n^2+\b}\right)^{\nu+1}}{\left(\sqrt{2n^2+\a}+\sqrt{2n^2+\b}\right)^{\nu+2}}
\cdot{}_2F_{1}\left[\nu+3/2,1/2;\nu+2;\left(\frac{\sqrt{2n^2+\a}-\sqrt{2n^2+\b}}{\sqrt{2n^2+\a}+\sqrt{2n^2+\b}} \right)^{2} \right].
\end{align*}
Replacing $\a$ by $2\a$ and $\b$ by $2\b$, we find that
\begin{align}\label{an1}
	&\frac{1}{4(\nu+1)}\left( \frac{\sqrt{\a}-\sqrt{\b}}{\sqrt{\a}+\sqrt{\b}}\right)^{\nu+1}+\sum_{n=1}^{\infty}I_{\nu+1}\left(\pi n\left(\sqrt{\a}-\sqrt{\b}\right)\right)K_{\nu+1}\left(\pi n\left(\sqrt{\a}+\sqrt{\b}\right)\right)\nonumber\\
	=&\frac{\Gamma(\nu+3/2)}{2\sqrt{2\pi}\Gamma(\nu+2)} \frac{\left(\sqrt{\a}-\sqrt{\b}\right)^{\nu+1}}{\left(\sqrt{\a}+\sqrt{\b}\right)^{\nu+2}}\cdot
	{}_2F_{1}\left(\nu+3/2,1/2;\nu+2;\left( \frac{\sqrt{\a}-\sqrt{\b}}{\sqrt{\a}+\sqrt{\b}}\right)^{2} \right)\nonumber\\
	&+\frac{\Gamma(\nu+3/2)}{\sqrt{\pi}\Gamma(\nu+2)}
	\sum_{n=1}^{\infty}\frac{\left(\sqrt{n^2+\a}
		-\sqrt{n^2+\b}\right)^{\nu+1}}{\left(\sqrt{n^2+\a}+\sqrt{n^2+\b}\right)^{\nu+2}}
	\cdot{}_2F_{1}\left[\nu+3/2,1/2;\nu+2;\left(\frac{\sqrt{n^2+\a}-\sqrt{n^2+\b}}{\sqrt{n^2+\a}+\sqrt{n^2+\b}} \right)^{2} \right].
\end{align}
Dividing both sides by $(\sqrt{\a}-\sqrt{b})^{\nu+1}$, letting $\a\to\b$, multiplying both sides of the resulting identity by $2(\nu+1)\G(\nu+1)(2\sqrt{\b})^{\nu+1}$, replacing $\nu$ by $\nu-1$ and $\b$ by $z^2/(4\pi^2)$, and rearranging, for $\Re(z)>0$, we recover an important result of Watson \cite[Equation (4)]{watsonselfreciprocal}:
\begin{align*}
\frac{1}{2}\G(\nu)+2\sum_{n=1}^\infty \left(\frac{1}{2}nz\right)^\nu K_\nu(nz)&=\G\left(\frac{1}{2}\right)\G\left(\nu+\frac{1}{2}\right)z^{2\nu}\left\{\frac{1}{z^{2\nu+1}}+2
\sum_{n=1}^\infty \frac{1}{(z^2+4n^2\pi^2)^{\nu+\frac{1}{2}}}\right\}.
\end{align*}

We now provide a generalization of yet another identity of Watson \cite[Equation (6)]{watsonselfreciprocal}.

\begin{corollary}\label{corK}
Let $K(k)$ denote the complete elliptic integral of the first kind defined by
\begin{equation}\label{kk}
	K(k):=\int_{0}^{\pi/2}\frac{d\theta}{\sqrt{1-k^2\sin^{2}(\theta)}},\qquad 0\leq |k| <1.
\end{equation}
For $\Re(\sqrt{\alpha})>\Re(\sqrt{\beta})>0$,
\begin{align}\label{finalWatson}
	&\sum_{n=1}^{\infty}I_{0}\left(\pi n\left(\sqrt{\a}-\sqrt{\b}\right)\right)K_{0}\left(\pi n\left(\sqrt{\a}+\sqrt{\b}\right)\right)\notag\\=&\frac{1}{\pi\left(\sqrt{\a}+\sqrt{\b}\right)}K\left(\left( \frac{\sqrt{\a}-\sqrt{\b}}{\sqrt{\a}+\sqrt{\b}}\right)^{2}\right)
	+\frac{\gamma+\log\left(\sqrt{\
			\a}+\sqrt{\b} \right)-\log 4}{2}\notag\\&+\sum_{n=1}^{\infty}\left\{\frac{2}{\pi\left(\sqrt{n^2+\a}+\sqrt{n^2+\b}\right)} K\left(\left( \frac{\sqrt{n^2+\a}-\sqrt{n^2+\b}}{\sqrt{n^2+\a}+\sqrt{n^2+\b}}\right)^{2}\right)-\frac{1}{2n}\right\}.
\end{align}	
\end{corollary}

\begin{proof}
Corollary \ref{corK} follows by analytically continuing \eqref{an1} to the region $\Re(\nu)>-3/2$ and then letting $\nu\to-1$. Since the argument is similar to that given in Section 5 of \cite{bdkz}, we discuss it only briefly here.

 Let $g(n)$ denote the $n^{\text{th}}$ summand in the series on the right-hand side of \eqref{an1}.
   It is not difficult to show that, as $n\to\infty$,
\begin{align*}
	g(n)\sim\frac{(\a-\b)^{\nu+1}}{(2n)^{2\nu+3}}.
\end{align*}
Therefore, for  $\Re(\nu)>-1$,
\begin{equation*}
	\sum_{n=1}^{\infty}g(n)=\sum_{n=1}^{\infty}\left(g(n)-\frac{(\a-\b)^{\nu+1}}{(2n)^{2\nu+3}}\right)
+\frac{(\a-\b)^{\nu+1}}{2^{2\nu+3}}\zeta(2\nu+3).
	\end{equation*}
Substituting this in \eqref{an1} and rearranging, we find that, for $\Re(\nu)>-1$,
\begin{align}\label{beforelimit}
&\sum_{n=1}^{\infty}I_{\nu+1}\left(\pi n\left(\sqrt{\a}-\sqrt{\b}\right)\right)K_{\nu+1}\left(\pi n\left(\sqrt{\a}+\sqrt{\b}\right)\right)\nonumber\\
&=\frac{\Gamma(\nu+3/2)}{2\sqrt{2\pi}\Gamma(\nu+2)} \frac{\left(\sqrt{\a}-\sqrt{\b}\right)^{\nu+1}}{\left(\sqrt{\a}+\sqrt{\b}\right)^{\nu+2}}
{}_2F_{1}\left(\nu+3/2,1/2;\nu+2;\left( \frac{\sqrt{\a}-\sqrt{\b}}{\sqrt{\a}+\sqrt{\b}}\right)^{2} \right)\nonumber\\
&\quad+\frac{\sqrt{2}\Gamma(\nu+3/2)}{\sqrt{\pi}\Gamma(\nu+2)}
\sum_{n=1}^{\infty}\left(g(n)-\frac{(\a-\b)^{\nu+1}}{(2n)^{2\nu+3}}\right)\nonumber\\
&\quad+(\sqrt{\a}-\sqrt{\b})^{\nu+1}\left(\frac{\Gamma(\nu+3/2)}{\sqrt{\pi}\Gamma(\nu+2)}
\frac{(\sqrt{\a}+\sqrt{\b})^{\nu+1}}{2^{2\nu+3}}\zeta(2\nu+3)
-\frac{1}{4(\nu+1)\left(\sqrt{\a}+\sqrt{\b}\right)^{\nu+1}}\right).
\end{align}

Observe that both sides of \eqref{beforelimit} are analytic in Re$(\nu)>-2$ with a removable singularity at $\nu=-1$, because
\begin{gather}
\lim_{\nu\to-1}\left(\frac{\Gamma(\nu+3/2)}{\sqrt{\pi}\Gamma(\nu+2)}
\frac{(\sqrt{\a}+\sqrt{\b})^{\nu+1}}{2^{2\nu+3}}\zeta(2\nu+3)
-\frac{1}{4(\nu+1)\left(\sqrt{\a}+\sqrt{\b}\right)^{\nu+1}}\right)\notag\\
=\frac{\gamma}{2}+\frac{1}{2}\log(\sqrt{\a}+\sqrt{\b})-\log 2,\label{limitnu}
\end{gather}
which can be seen from expanding each side of \eqref{limitnu} in Taylor series about $\nu=-1$.
Thus, letting $\nu\to-1$ on both sides of \eqref{beforelimit} and using \eqref{limitnu}, we arrive at \eqref{finalWatson}, where we used the identity \cite[p.~908, Formula 8.113, no.~2]{gr}
\begin{equation*}
{_2F_1}\left(\frac12,\frac12;1;x^2\right)=\frac{2}{\pi}K(x)
\end{equation*}
where $K(x)$ is defined in \eqref{kk}.
\end{proof}

As previously indicated, the identity \eqref{finalWatson} is a generalization of the following identity of Watson \cite{watsonselfreciprocal}.

\begin{corollary}\label{watsoncor} For
 $\Re( \beta)>0$,
\begin{equation}\label{wat}
2\sum_{n=1}^{\infty}K_{0}(n\beta)=\pi\left\{\frac{1}{\beta}
+2\sum_{n=1}^{\infty}\left(\frac{1}{\sqrt{\beta^2+4\pi^2n^2}}
-\frac{1}{2n\pi}\right)\right\}+\gamma+\log\left(\frac{\beta}{2}\right)-\log 2\pi.
\end{equation}
\end{corollary}

\begin{proof} If we let $\a\to\b$ in \eqref{finalWatson} and use the trivial facts
\begin{equation*}
\lim_{\a\to\b^{+}}I_{0}(\pi(\sqrt{n\a}-\sqrt{n\b}))K_{0}(\pi(\sqrt{n\a}+\sqrt{n\b}))
=K_{0}(2\pi\sqrt{n\b})
\end{equation*}
and $K(0)=\tfrac12 \pi$,
we obtain \eqref{wat}.
\end{proof}

\begin{remark}
Each of the identities  \eqref{c1}, \eqref{chifinal}, and \eqref{c6} can be analytically continued in the same manner as that for \eqref{an1} for Corollary \ref{corK}.
\end{remark}
Letting $\nu=-1/2$ in \eqref{an1}, and using \eqref{I1}, \eqref{I2}, and \eqref{log}, we obtain
\begin{align}\label{loglast}
&\frac{1}{2}\left(\frac{\sqrt{\a}-\sqrt{\b}}{\sqrt{\a}+\sqrt{\b}}\right)^{1/2}
+\frac{\sqrt{2}}{\pi\sqrt{\a-\b}}\sum_{n=1}^{\infty}
\df{e^{-\frac{\pi n}{\sqrt{2}}\left(\sqrt{\a}+\sqrt{\b}\right)}}{n}\sinh\left(\frac{\pi n}{\sqrt2}\left(\sqrt{\a}-\sqrt{\b}\right)\right)\nonumber\\
&=\frac{1}{2\sqrt{2}\pi\sqrt{\a-\b}}\log\left( \frac{\a}{\b}\right)+\frac{1}{\sqrt{2}\pi\sqrt{\a-\b}}\sum_{n=1}^{\infty}\log\left(\frac{2n^2+\a}{2n^2+\b}\right).
\end{align}
A rearrangement of  \eqref{loglast} leads to
\begin{gather}
\frac{\pi}{2}(\sqrt{\a}-\sqrt{\b})+\sqrt2\sum_{n=1}^{\infty}\frac{e^{-\frac{\pi n}{\sqrt2}\left(\sqrt{\a}+\sqrt{\b}\right)}}{n}\sinh\left(\frac{\pi n}{\sqrt2}\left(\sqrt{\a}-\sqrt{\b}\right)\right)\notag\\=\frac{1}{2\sqrt{2}}\log\left( \frac{\a}{\b}\right)+\frac{1}{\sqrt{2}}\sum_{n=1}^{\infty}\log\left(\frac{2n^2+\a}{2n^2+\b}\right).\label{c222}
\end{gather}
Using the elementary Maclaurin series
\begin{equation*}
-\log(1-x)=\sum_{k=1}^{\infty}\df{x^k}{k},\quad |x|<1,
\end{equation*}
in \eqref{c222}, we conclude that
\begin{align*}
&\frac{\pi}{2}(\sqrt{\a}-\sqrt{\b})+\frac{1}{\sqrt2}\log\left(\frac{1-e^{-\pi \sqrt{2\a}}}{1-e^{-\pi \sqrt{2\b}}} \right)=\frac{1}{2\sqrt2}\log\left( \frac{\a}{\b}\right)+\frac{1}{\sqrt{2}}\sum_{n=1}^{\infty}\log\left(\frac{2n^2+\a}{2n^2+\b}\right).
\end{align*}

\begin{center}
\textbf{Acknowledgements}
\end{center}

The first and second authors sincerely thank the MHRD SPARC project SPARC/2018-2019/P567/SL for their financial support.  The first author is also supported by a grant from the Simons Foundation.  The third author is a postdoctoral fellow at IIT Gandhinagar supported, in part, by the grant CRG/2020/002367

\end{document}